\documentclass[11pt, a4paper,reqno]{amsart}

\usepackage{subcaption}

\usepackage{color} \definecolor{bleu_sombre}{rgb}{0,0,0.6}  \definecolor{rouge_sombre}{rgb}{0.8,0,0}\definecolor{vert_sombre}{rgb}{0,0.6,0}
\usepackage[plainpages=false,colorlinks,linkcolor=bleu_sombre,citecolor=rouge_sombre,urlcolor=vert_sombre,breaklinks]{hyperref}

\usepackage[english]{babel}

\usepackage{amsmath,amssymb,amsthm,graphicx,amsfonts,url,color,enumerate,dsfont,stmaryrd,mathabx, mathtools}

\theoremstyle{plain}
\newtheorem{theorem}{{Theorem}}[section] 
\newtheorem*{theorem*}{{Theorem}}
\newtheorem{proposition}[theorem]{Proposition}
\newtheorem*{proposition*}{Proposition}
\newtheorem{corollary}[theorem]{Corollary}
\newtheorem*{corollary*}{Corollary}
\newtheorem{lemma}[theorem]{Lemma}
\newtheorem*{lemma*}{Lemma}

\theoremstyle{definition}

\newtheorem*{definition*}{Definition}

\theoremstyle{remark}

\makeatletter

\@addtoreset{equation}{section}  
\makeatother

\newcommand{\inner}[2]{\left\langle{#1}, {#2}\right\rangle} 
\newcommand{\seq}[1]{\left\{#1_n\right\}_{n = 1}^{\infty}}

\newcommand{\dd}{\,\textrm{d}}
\newcommand{\norm}[1]{\left|\left|#1\right|\right|}
\newcommand\ran{\mathop\mathrm{ran}\nolimits}

\newcommand{\abs}[1]{\left\vert #1\right\vert}

\renewcommand{\leq}{\leqslant}	\renewcommand{\geq}{\geqslant}

\renewcommand{\Re}{\mathsf{Re}\,}        
\renewcommand{\Im}{\mathsf{Im}\,}
\newcommand{\Tr}{\mathsf{Tr}\,}

\newcommand{\Dom}{\mathsf{dom}\,}

\renewcommand{\ker}{\mathsf{ker}} 
\newcommand{\Ran}{\mathsf{ran}\,}

\renewcommand{\deg}{\mathsf{deg}\,}

\usepackage{mathrsfs}

\newcommand{\detail}[1]
{
}

\begingroup
    \makeatletter
    \@for\theoremstyle:=definition,remark,plain\do{%
        \expandafter\g@addto@macro\csname th@\theoremstyle\endcsname{%
            \addtolength\thm@preskip\parskip
            }%
        }
\endgroup
\usepackage{color}
\usepackage[normalem]{ulem}
\definecolor{DarkBlue}{rgb}{0,0.1,0.7}  
\definecolor{DarkGreen}{rgb}{0,0.5,0.1}

\newcommand\soutD{\bgroup\markoverwith
{\textcolor{DarkGreen}{\rule[.5ex]{2pt}{1pt}}}\ULon}
\newcommand\soutJ{\bgroup\markoverwith
{\textcolor{DarkBlue}{\rule[.5ex]{2pt}{1pt}}}\ULon}
\newcommand{\Hm}[1]{\leavevmode{\marginpar{\tiny%
$\hbox to 0mm{\hspace*{-0.5mm}$\leftarrow$\hss}%
\vcenter{\vrule depth 0.1mm height 0.1mm width \the\marginparwidth}%
\hbox to
0mm{\hss$\rightarrow$\hspace*{-0.5mm}}$\\\relax\raggedright #1}}}

\title[The wave equation with Dirac damping on a compact star graph]{Spectrum of the wave equation with Dirac damping on a compact star graph}

\author{Mikul\'{a}\v{s} Ku\v{c}era}
\address[Mikul\'{a}\v{s} Ku\v{c}era]{Department of Mathematics, Faculty of Nuclear Sciences and Physical Engineering, Czech Technical University in Prague, Trojanova 13, 120 00, Prague, Czech Republic}
\email{kucermik@fjfi.cvut.cz}

\date{\today}

\begin{document}

\begin{abstract}
We consider the wave equation with a distributional Dirac damping and Dirichlet boundary conditions on a compact interval. It is shown that the spectrum of the corresponding wave operator is fully determined by the zeros of an entire function. Consequently, a considerable change of spectral properties is shown for certain critical values of the damping parameter. We also derive a definitive criterion for the Riesz basis property of the root vectors for an arbitrary placement of a complex-valued Dirac damping. For irrational damping positions, we identify an additional measure-zero set of exceptional damping parameters for which the root vectors are complete but fail to form a Riesz basis. Finally, we consider a generalization of the problem to compact star graphs and provide insight into the essence of the critical damping constant.
\end{abstract}

\maketitle
 
\section{Introduction}
%
To model playing harmonics on a string, in \cite{BRT}, Bamberger, Rauch, and Taylor introduced the following wave equation:
\begin{equation}
u_{tt}(x,t) - u_{xx}(x,t) + \alpha \delta(x-a)u_t(x,t) = 0, \quad x \in [0, \pi], \quad t \geq 0,
    \label{WaveEquation}
\end{equation}
where $u:\left[0, \pi\right] \times [0, +\infty) \to \mathbb C$ is the displacement of the string, $a \in  (0, \pi)$, subject to the Dirichlet boundary conditions
\[
u(0, t) = 0 = u(\pi, t), \quad t \geq 0.
\]
By detailed analysis of the contraction semigroup of the corresponding wave operator acting on the Hilbert space $\mathcal H = \dot H_0^1(0, \pi) \times L^2(0, \pi)$, the authors show that in the case of central placement $a = \pi/2$, the optimal damping constant (i.e. such that ensures the fastest decay of non-harmonic modes) is $\alpha = 2$.

Further research was carried out in \cite{CH} by Cox and Henrot. Using the `shooting function' method, they characterised the eigenvalues as roots of an entire function. It was also established that in the special case $a = p\pi/q$, with $p$ and $q$ being coprime integers, and $\alpha \in [0, +\infty) \setminus \left\{2\right\}$, the root vectors comprise a Riesz basis for $\mathcal H$. However, the analysis remained incomplete for $\alpha = 2$ and arbitrary placement of the damping as well as a general complex damping $\alpha \in \mathbb C$ which has been recently considered in \cite{KrLip} in the present model and in \cite{KrRoy} for the wave equation on the real line.

For the broader context of wave propagation, observation and control on networks, we refer to Dáger and Zuazua \cite{DagerZuazua}; see also Valein and Zuazua \cite{ValeinZuazua} for stabilization through weighted observability estimates. These works and the references therein highlight how network geometry and spectral resonances affect observability and can obstruct stabilization. The present paper complements this literature by providing precise spectral and Riesz-basis characterizations for point-damped intervals and equilateral compact star graphs, allowing complex damping parameters.

 This paper aims to solidify the `shooting function' method in the complex setting. Reasoning with the poles of the resolvent in the spirit of \cite{CoxZuazua}, it is shown that algebraic multiplicities of eigenvalues are exactly their multiplicities as roots of the function in question. For rational placement of the damping, this characterisation consequently allows us to determine whether or not the root vectors form a Riesz basis in $\mathcal H$. The importance of this result lies in the fact that it enables the simple spectral solution of \cite{CoxZuazua} to the optimal damping problem proposed in \cite{BRT} and mentioned above.

 Similarly to the non-compact star graph model of Krej\v{c}i\v{r}\'\i k and Royer \cite{KrRoy}, using a simple symmetry observation, we extend the results to admit arbitrary complex Dirac dampings. The generalization to an unrestricted placement is then achieved using a result of Krej\v{c}i\v{r}\'\i k and Lipovsk\'y \cite{KrLip}. The authors calculated the spectral determinant of the wave operator and, as expected, confirmed its singular behaviour at $\alpha = \pm 2$. In the case of an irrational damping placement, we identify an additional set of damping values that lead to the root vectors being complete but not comprising a Riesz basis.

Finally, some insight is provided into the appearance of the values $\alpha = \pm 2$ as critical points of the model. Following the footsteps of \cite{KrRoy}, we analyse the wave equation on a compact star graph in the sense of \cite{Kurasov} with $n \in \mathbb N$ edges. It is shown that the abrupt change of spectral properties happens precisely for $\alpha = \pm n$.

This behaviour of the wave equation subject to a non-regular damping is not unprecedented. We refer to the 2020 article \cite{Siegl}. The authors consider the singular damping of the form
\[
u_{tt}(x,t) - u_{xx}(x,t) + \frac{\alpha}{x}u_t(x,t) = 0, \quad x \in (0, 1), \quad t \geq 0,
\]
with Dirichlet boundary conditions and $\alpha > 0$. It was shown that the otherwise infinite spectrum suddenly shrinks to $n-1$ eigenvalues whenever $\alpha = 2n$ for $n \in \mathbb N$. In particular, for $\alpha = 2$, the spectrum becomes empty. Moreover, in this concrete setting with $\alpha = 2$, all solutions of the wave equation are shown to vanish at finite time.

The present model has another, perhaps less apparent, possible application reaching into relativistic quantum mechanics. The traditional spectral approach to the wave equation \eqref{WaveEquation} lies in considering $\psi = (u, u_t)^T $ and rewriting it as
\[
A(a, \alpha)\psi = \partial_t\psi, \quad A(a, \alpha) = \begin{pmatrix}
    0 & I \\ \partial_{xx} & -\alpha\delta_a
\end{pmatrix},
\]
where $A(a,\alpha)$ is the generator of the corresponding semigroup. However, one can take $\phi = (u_t, u_x)^T$ and rewrite \eqref{WaveEquation} as 
\[
\mathrm i D(a, \alpha)\phi = \partial_t\phi, \quad D(a, \alpha) = \begin{pmatrix}
    \mathrm {i}\alpha\delta_a & -\mathrm{i}\partial_x \\
    -\mathrm{i}\partial_x & 0
\end{pmatrix}.
\]
The one-dimensional Dirac-type operator with a single-component highly localized potential $D(a,\alpha)$ is self-adjoint whenever $\alpha \in \mathrm{i}\mathbb R$. As discovered in \cite{Gesztesy1, Gesztesy2}, unitary equivalence exists between $A(a, \alpha)$ acting in $\dot H_0^1(0, \pi) \times L^2(0, \pi)$ and $\mathrm{i}D(a, \alpha)$ considered in $L^2(0,\pi) \times L^2(0, \pi)$. This observation  further emphasizes the importance of accounting for a complex damping parameter.

To put this reformulation into a broader context, we refer to the recent work \cite{HeribanTusek} for a detailed investigation of one-dimensional Dirac operators with non-self-adjoint point interactions. Related self-adjoint $\delta$-shell interactions supported on straight lines have likewise been studied and proposed as continuum models for line defects in Dirac materials (see \cite{BHT}).

The paper is organised as follows. Section~\ref{Sec.results} introduces the model and its basic known properties. Our main results concerning the spectrum and basis of root vectors are also formulated here. Section~\ref{Sec.general} provides proofs and outlines of calculations of the general properties of the wave operator needed for our analysis. The results concerning Riesz basis of root vectors are proven in Section~\ref{sec.Basis}. In Section~\ref{CompactStarGraph}, we explain the appearance of the mysterious damping constant $\pm 2$ by considering the wave equation with Dirac damping on a compact star graph.

\section{The model and main results}\label{Sec.results}
%

\subsection{The damped wave equation and the wave operator}
Our setting is the Hilbert space $\mathcal H = \dot H_0^1(0, \pi) \times L^2(0, \pi)$ endowed with the inner product
\[
\inner{\phi}{\psi} = \inner{\phi_1'}{\psi_1'}_{L^2} + \inner{\phi_2}{\psi_2}_{L^2}.
\]

Setting $\psi = (u, u_t)$, the wave equation \eqref{WaveEquation} can be reformulated as
\begin{gather}
A(a, \alpha)\psi = \psi_t, \quad A(a,\alpha) = \begin{pmatrix} 0 & I \nonumber \\ \partial_{xx} & 0 \end{pmatrix}, \quad \psi(x, 0) = \psi_0(x), \\ \Dom A(a,\alpha) =\nonumber \\ \left\{\psi \in \left(\dot H_0^1(0, \pi) \cap H^2(0, a) \cap H^2(a, \pi)\right) \times \dot H_0^1(0, \pi) \mid \psi_1'(a+) - \psi_1'(a-) = \alpha\psi_2(a) \right\}
\label{DefinitionOfA}
\end{gather}
with Cauchy data $\psi(x, 0) = \psi_0(x)$, where $\psi \in \mathcal H$. The jump condition in the domain is to be understood in the sense of the absolutely continuous representative in the given equivalence class of $\dot H_0^1(0, \pi)$.

It is shown in \cite{BRT} that the operator $A(a, \alpha)$ defined above is maximally dissipative for $\alpha \geq 0$. It then follows from the Lumer-Phillips theorem \cite[Theorem 4.3]{Pazy} that $A$ gives rise to the contraction semigroup $\exp(tA)$ which is a suitable solution to \eqref{WaveEquation}.
By \cite{BRT}, the operator has a compact resolvent for all values $a \in (0, \pi)$ and $\alpha \in \mathbb C$; therefore, it also has purely discrete spectrum. We will replicate and improve the result by showing that the inverse is even Hilbert-Schmidt and computing its Hilbert-Schmidt norm in Section~\ref{Sec.general}.

It was further discovered by Bamberger, Rauch, and Taylor that the harmonic spectrum (i.e. purely imaginary eigenvalues) is non-empty if and only if $a$ is a rational multiple of $\pi$. They also found out that all eigenvalues are geometrically simple. When it comes to determining algebraic multiplicity, the following result is pivotal.

\begin{theorem}\label{CharacteristicFunction}
    Let $\alpha \in \mathbb C$ and $a \in (0,\pi)$. Then $\lambda \in \mathbb C$ is an eigenvalue of $A(a, \alpha)$ if and only if it is a root of the entire function
    \begin{equation}
    S(\lambda; a, \alpha) \coloneqq \frac{1}{\lambda}\left(\sinh(\lambda\pi) + \alpha \sinh(\lambda a)\sinh(\lambda(\pi-a))\right).
        \label{DefinitionOfS}
    \end{equation}

    Additionally, the algebraic multiplicity of the eigenvalue $\lambda$ is exactly its multiplicity as a root of $S(\lambda; a, \alpha)$.
\end{theorem}
As a consequence, all eigenvalues are shown to be algebraically at most double by Corollary \ref{C: Algebraic multiplicity}.

Since the analysis for $\alpha > 0$, $\alpha \neq 2$ has to some extent been done by Cox and Henrot in \cite{CH}, we shall henceforth focus especially on the case $\alpha = 2$ as well as, perhaps most notably, arbitrary $\alpha \in \mathbb C$.

\subsection{The root vectors}
Recall that a sequence in a Hilbert space is called a Riesz basis if it is the image of an orthonormal basis under a bounded isomorphism. In pursuit of proving or disproving the Riesz basis property for any $\alpha \in \mathbb C$ and $a \in (0, \pi)$, we first note the following symmetrical relationship between $A$ and its adjoint:

\begin{proposition}\label{AdjointProposition}
    For any $a \in (0, \pi)$ and $\alpha \in \mathbb C$, the adjoint operator of $A(a, \alpha)$ is $A^\ast(a, \alpha) = -A(a, -\overline{\alpha})$.
\end{proposition}

Since both of the operators $A$ and $A^\ast$ possess compact resolvents (by combining Proposition \ref{AdjointProposition} and Theorem \ref{T: HS norm}), their systems of root vectors can be ordered and normalized to form biorthogonal sequences. In other words, if $\seq{\psi}$ is the system of root vectors of $A$ and $\seq{\phi}$ the system of root vectors of $A^\ast$, we can demand $\inner{\phi_m}{\psi_n} = \delta_{mn}$. This observation allows us to determine the Riesz basis property based on the following criterion.

\begin{theorem} \cite[Ch. 1, Theorem 9]{Young}. \label{RieszCriterion}
    Let $\mathcal H$ be a Hilbert space and $\seq{\psi} \subset \mathcal H$. Then $\seq{\psi}$ is a Riesz basis in $\mathcal H$ if and only if it is complete in $\mathcal H$ and Bessel and possesses a biorthogonal sequence $\seq{\phi}$ that is also complete and Bessel.
\end{theorem}

Recall that by \cite{Young} a sequence $\seq{\psi}$ in a Hilbert space $\mathcal H$ is Bessel if and only if 
\[
\sum_{n=1}^\infty\left|\inner{\psi_n}{\psi}\right|^2 < + \infty, \quad \forall \psi \in \mathcal H.
\]
A subset of $\mathcal H$ is said to be complete if its linear span is dense in $\mathcal H$.

The verification of the condition in Theorem \ref{RieszCriterion} was done in \cite{CH} for positive $\alpha \neq 2$ and rational placement of the damping $a = p\pi/q$. In Section~\ref{sec.Basis} we show the derivation of the general condition:

\begin{theorem}\label{Criterion}
    Let $a \in (0, \pi)$ and $\alpha \in \mathbb C$. The root vectors of $A(a, \alpha)$ form a complete sequence in $\mathcal H$ if and only if $\alpha \neq \pm 2$.
    \begin{enumerate}
        \item If $a \in \pi \mathbb Q$, this sequence is a Riesz basis.
        \item If $a \notin \pi\mathbb Q$, then there is a measure-zero set $\mathcal E_a \subset \mathbb C\setminus \{\pm 2\}$ such that for any $\alpha \in \mathcal E_a$, the sequence is complete in $\mathcal H$, but it does not comprise a Riesz basis. For $\alpha \notin \mathcal E_a \cup \{\pm 2\}$, the sequence is a Riesz basis.
    \end{enumerate}
\end{theorem}
For the precise characterisation of the set $\mathcal E_a$, we refer to Theorem \ref{T: Irrational Riesz criterion}. For its illustration, see Figure \ref{fig: Exceptional set}.

\begin{figure}[h!]
    \centering
    \begin{subfigure}{0.37\textwidth}
        \includegraphics[width=\linewidth]{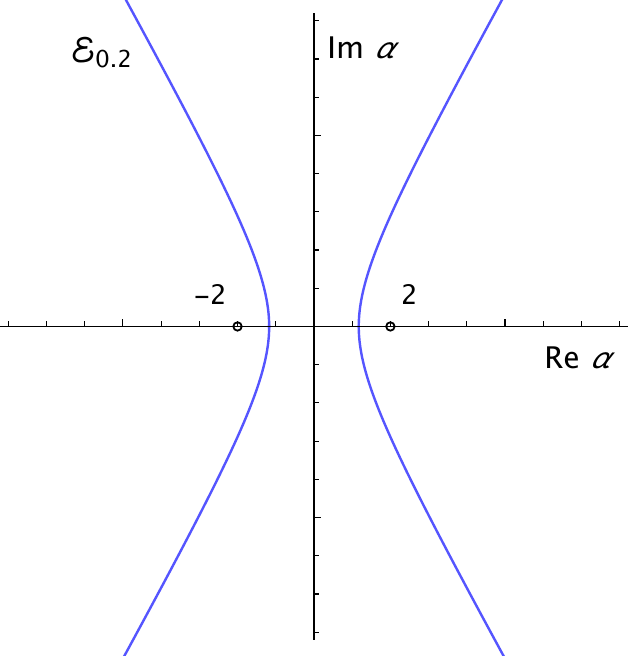}
        \caption{The set $\mathcal E_a$ for $a = 0.2$.}
    \end{subfigure}
    \hfill
    \begin{subfigure}{0.37\textwidth}
        \includegraphics[width=\linewidth]{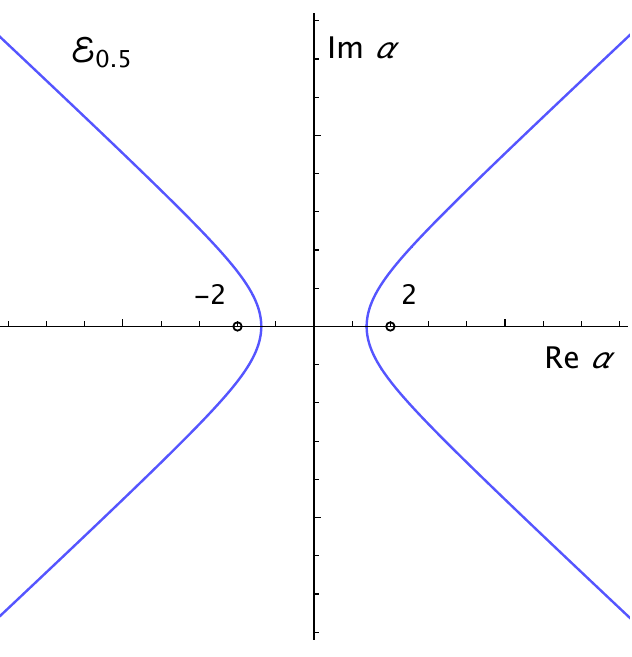}
        \caption{The set $\mathcal E_a$ for $a = 0.5$.}
    \end{subfigure}
    \medskip 
    
    \begin{subfigure}{0.37\textwidth}
        \includegraphics[width=\linewidth]{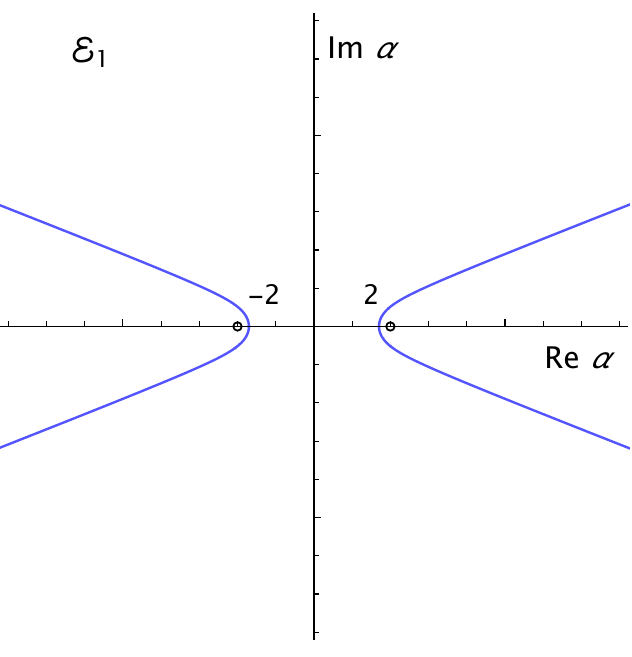}
        \caption{The set $\mathcal E_a$ for $a = 1$.}
    \end{subfigure}
    \hfill    
    \begin{subfigure}{0.37\textwidth}
        \includegraphics[width=\linewidth]{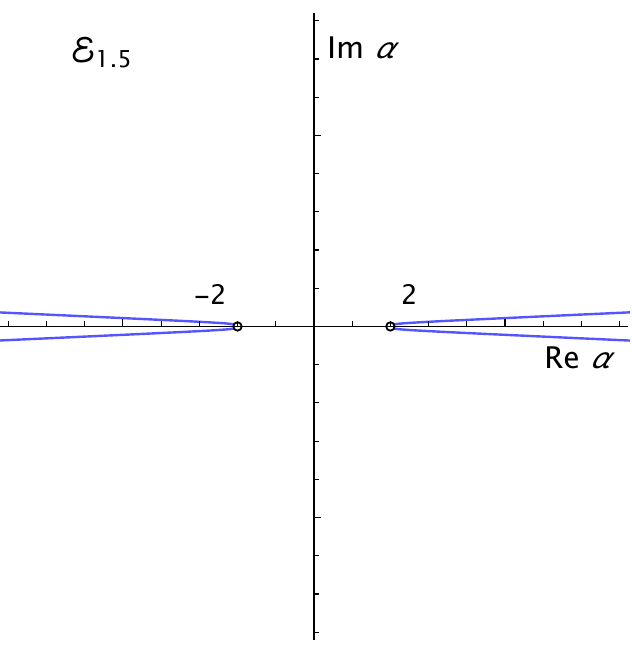}
        \caption{The set $\mathcal E_a$ for $a = 1.5$.}
    \end{subfigure}
     \caption{An illustration of the exceptional set $\mathcal E_a$ (blue curves) from Theorem \ref{Criterion} for four different values of $a$.}
    \label{fig: Exceptional set}
\end{figure}
\newpage

\subsection{The star graph setting}
The wave equation on an interval can be viewed as a special case of the wave equation on a compact star graph. Taking $n \in \mathbb N$ copies of the interval $[0, \pi]$ subject to Dirichlet boundary conditions at $\pi$ and the distributional damping $\alpha\delta$ placed at the common central vertex $0$, one can define the generator of the wave equation as a closed operator acting in a suitable Hilbert space. See Figure \ref{fig: Star Graph} for an illustration. An analysis analogous to that for the interval yields the following principal result: 
\[
\text{the root vectors form a Riesz basis}\quad \iff \quad\alpha \neq \pm n;
\]
that is, the critical values of the damping parameter are precisely $\pm n$, where $n$ is the number of edges adjacent to the damped vertex. In the case $\alpha = \pm n$, the root vectors cease to be complete. For a rigorous formulation and further details, we refer the reader to Section \ref{CompactStarGraph}.

\begin{figure}[ht]
    \centering
    \includegraphics[width=0.5\linewidth]{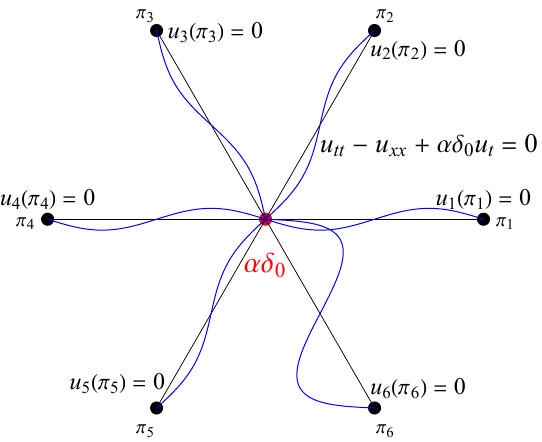}
    \caption{An illustration of the wave equation on a compact star graph. The outer vertices are subject to Dirichlet boundary conditions; the Dirac damping is placed at the central vertex.}
    \label{fig: Star Graph}
\end{figure}

\section{General properties of the wave operator}\label{Sec.general}
\subsection{The adjoint operator} For further use, it is convenient to know the formula for the adjoint of $A$. We strive to prove Proposition \ref{AdjointProposition}.
Let $\phi \in \Dom A^\ast$, $\eta \coloneqq A^\ast\phi$. Then we have
\begin{equation}
\inner{\phi_1'}{\psi_2'}_{L^2} + \inner{\phi_2}{\psi_1''}_{L^2} = \inner{\phi}{A\psi} = \inner{\eta}{\psi} = \inner{\eta_1'}{\psi_1'}_{L^2} + \inner{\eta_2}{\psi_2}_{L^2}
\label{AdjointDerivation}
\end{equation}
for any $\psi \in \Dom A$.

Setting $\psi_1 = 0$ and $\psi_2 \in C_0^\infty(0, a)$ or $C_0^\infty(a, \pi)$ in \eqref{AdjointDerivation} yields $\phi_1' \in \dot H^1(0, a) \cap \dot H^1(a, \pi)$ (the choice is consistent with the domain of $A$ using the fact that $\psi_2(a) = 0$). Next, choose $\psi_1 \in H^2(0, \pi)$ and $\psi_2 = 0$ (the choice is again consistent with $\Dom A$ since the derivative of $\psi_1$ is continuous on $(0, \pi)$ whenever $\psi_2(a) = 0$). Hence, we have
\[
\inner{\phi_2}{\psi_1''}_{L^2} = -\inner{\eta_1}{\psi_1''}_{L^2}.
\]
Making use of the surjectivity of the Dirichlet Laplacian on $(0, \pi)$ and the fact that the choices of $\psi_1$ cover precisely its domain, we conclude that $\phi_2 = - \eta_1$ in the distributional sense and thus also almost everywhere. It follows that $\phi_2 \in \dot H_0^1(0, \pi)$.

For any $\phi \in \Dom A^\ast$ and $\psi \in \Dom A$, we can now calculate:

\begin{align*}
    \inner{A^\ast\phi}{\psi} &= \inner{\phi}{A\psi} = \inner{\phi_1'}{\psi_2'}_{L^2} + \inner{\phi_2}{\psi_1''}_{L^2} = -\psi_2(a)\left(\overline{\phi_1'(a+)} - \overline{\phi_1'(a-)}\right) \\ &\quad- \overline{\phi_2(a)}\left(\psi_1'(a+) - \psi_1'(a-)\right) - \inner{\phi_1''}{\psi_2}_{L^2} - \inner{\phi_2'}{\psi_1'}_{L^2} \\
    &= -\psi_2(a)\left(\overline{\phi_1'(a+)} -\overline{\phi_1'(a-)} + \alpha\overline{\phi_2(a)}\right) + \inner{\begin{pmatrix}
        0 & -I \\ -\partial_{xx} & 0 
    \end{pmatrix}\begin{pmatrix}
        \phi_1 \\ \phi_2
    \end{pmatrix}}{\begin{pmatrix}
        \psi_1 \\ \psi_2
    \end{pmatrix}}.
\end{align*}
This already forces $\phi_1'(a+) - \phi_1'(a-) = - \overline{\alpha} \phi_2(a)$ and consequently, indeed $A^\ast(a, \alpha) = -A(a, -\overline{\alpha})$. We have proven Proposition \ref{AdjointProposition}.

\subsection{The resolvent}\label{SecResolvent}
Next, we provide an explicit construction of the resolvent. Consider the equation
\[
(A-\lambda I)\begin{pmatrix}
    u \\ v
\end{pmatrix} = \begin{pmatrix}
    f \\ g
\end{pmatrix}
\]
for some $\lambda \in \mathbb C$, $(u, v)^T \in \Dom A$ and $(f, g)^T \in \mathcal H$.
This gives us $v = \lambda u + f$ and $u'' -(\lambda + \alpha \delta_a)v = g$. Substituting the first equation into the second, we have the Sturm-Liouville problem
\begin{equation}
    u'' - \lambda(\lambda + \alpha\delta_a)u = (\lambda + \alpha\delta_a)f + g.
    \label{Resolventu}
\end{equation}
The approach is standard. We choose solutions $u_1$ and $u_2$ to \eqref{Resolventu} with zero right-hand side such that
\[
u_1(0) = 0 = u_2(\pi), \quad u_1'(0) = 1, \quad u_2'(\pi) = -1.
\]
The Green function is of the form
\begin{equation}
\mathcal G_\lambda(x,y) = -\frac{1}{u_1(\pi)} \begin{cases} u_1(x)u_2(y), & 0 \leq x \leq y \leq \pi, \\ 
u_1(y)u_2(x), & 0 \leq y \leq x \leq \pi,    
\end{cases}
\label{eq: Green function}
\end{equation}
where the denominator $-u_1(\pi)$ is the constant Wronskian of the associated homogeneous equation. The solution to \eqref{Resolventu} is then given by
\begin{equation}
u(x) = \int_0^\pi\mathcal G_\lambda(x,y)[(\lambda + \alpha\delta_a(y))f(y) + g(y)]\dd y \eqqcolon G_\lambda\left[(\lambda +\alpha\delta_a)f + g\right](x),
\label{GreenOperator}
\end{equation}
where $G_\lambda$ is the corresponding Green operator. Combined with the equation for $v$, we arrive at the following result.
\begin{proposition}
Let $a \in (0, \pi)$, $\alpha \in \mathbb C$, $\lambda \in \rho(A(a, \alpha))$. The resolvent of $A$ in $\lambda$ is
\begin{equation}
(A-\lambda I)^{-1} = \begin{pmatrix}
    G_\lambda(\lambda + \alpha\delta_a) & G_\lambda \\ I + \lambda G_\lambda(\lambda + \alpha\delta_a) & \lambda G_\lambda
\end{pmatrix},
    \label{Resolvent}
\end{equation}
where $G_\lambda$ is the Green operator defined in \eqref{GreenOperator}.
\end{proposition}

We will now show that $A^{-1}$ is Hilbert-Schmidt by directly calculating its Hilbert-Schmidt (HS) norm. Note that setting $\lambda = 0$ in \eqref{Resolvent}, the inverse simplifies to
\[
A^{-1} = \begin{pmatrix}
    G_0(\alpha\delta_a) & G_0 \\ I & 0
\end{pmatrix}.
\]

To calculate the HS norm, we will use the orthonormal basis of eigenfunctions of the unbounded operator $A_0$ with $\alpha = 0$. In such case, the operator is skew-adjoint, as follows from Proposition~\ref{AdjointProposition}.

We have
\[
A_0 = \begin{pmatrix}
    0 & I \\  \partial_{xx} & 0
\end{pmatrix},\quad \Dom A_0 = H_0^2(0, \pi) \times \dot H_0^1(0, \pi).
\]
Solving 
\[
A_0\begin{pmatrix}
    u \\ v
\end{pmatrix} = \lambda \begin{pmatrix}
    u \\ v
\end{pmatrix}
\]
gives $v = \lambda u$ and so $u'' - \lambda^2 u = 0$ with Dirichlet boundary conditions. 
Therefore,
\[
\lambda_n = \mathrm{i}n, \quad
\omega_n(x) = \frac{1}{n\sqrt{\pi}}\sin(nx)\begin{pmatrix}
    1 \\ \mathrm{in}
\end{pmatrix}, \quad n \in \mathbb Z \setminus \{0\},
\]
where $\omega_n$ are the normalized eigenfunctions. It is straightforward to check that the system $\left\{\omega_n\right\}_{n\in\mathbb Z \setminus \{0\}}$ forms an orthonormal basis in $\mathcal H$. It suffices to use the orthonormal bases of $L^2(0,\pi)$ formed by sines and cosines, respectively, and verify Parseval's inequality for $\left\{\omega_n\right\}_{n\in\mathbb Z \setminus \{0\}}$ in $\mathcal H$.

Back to the HS norm calculation; clearly
\[
(A^{-1}\omega_n)(x) = \frac{1}{n\sqrt{\pi}}\begin{pmatrix}
    \alpha \sin(na) \mathcal G_0(x, a) + \mathrm{i}n \int_0^\pi \mathcal G_0(x,y) \sin(ny) \dd y \\
    \sin(nx)
\end{pmatrix}.
\]
To find $\mathcal G_0$, we must solve for $u_1$ and $u_2$. A simple calculation yields $u_1(x) = x$, $u_2(x) = \pi -x$. Therefore,
\[
\mathcal G_0(x,y) = -\frac{1}{\pi}\begin{cases}
    x(\pi-y), & 0 \leq x \leq y \leq \pi, \\
    y(\pi-x), & 0 \leq y \leq x \leq \pi.
\end{cases}
\]

Consequently,
\[
\partial_x\mathcal G_0(x,y) = \begin{cases}
    -\frac{\pi-y}{\pi}, & 0 \leq x < y \leq \pi, \\
    \frac{y}{\pi}, & 0 \leq y < x \leq \pi
\end{cases}
\]
and
\[
\int_0^\pi \partial_x\mathcal G_0(x,y)\sin(ny) \dd y = \frac{1}{\pi}\left(\int_0^x y\sin(ny) \dd y -\int_x^\pi (\pi-y)\sin(ny)\dd y \right) = -\frac{\cos(nx)}{n}.
\]
We obtain
\[
\left(A^{-1}\omega_n\right)_1'(x) = \frac{1}{n\sqrt{\pi}}\left[\alpha\sin(na)\partial_x \mathcal G_0(x,a) - \mathrm{i}\cos(nx)\right].
\]
The $\dot H_0^1$ norm is bounded by
\begin{align*}
\norm{\left(A^{-1}\omega_n\right)_1'}^2_{L^2} &\leq \frac{1}{\pi n^2}\int_0^\pi |\alpha|^2\sin^2(na)(\partial_x \mathcal G_0(x,a))^2 + \cos^2(nx) \dd x\\
&= \frac{1}{\pi n^2} \left(\abs{\alpha}^2\frac{a(\pi-a)}{\pi}\sin^2(na) + \frac{\pi}{2}\right).
\end{align*}
Note that for $\alpha \in \mathbb R$, we obtain equality.

Overall, we have
\begin{align*}
\norm{A^{-1}\omega_n}^2 &= \norm{(A^{-1}\omega_n)_1'}^2_{L^2} + \norm{(A^{-1}\omega_n)_2}^2_{L^2} \leq \frac{1}{\pi n^2}\left(\abs{\alpha}^2\frac{a(\pi-a)}{\pi}\sin^2(na) + \pi\right)
.
\end{align*}

Taking the Fourier expansion of $x(\pi-x)$ on $(0, \pi)$, one has
\[
x(\pi-x) = \frac{\pi^2}{6} - \sum_{n=1}^\infty\frac{\cos(2nx)}{n^2} = 2\sum_{n=1}^\infty \frac{\sin^2(nx)}{n^2}.
\]

Finally, the HS norm is
\begin{equation*}
\norm{A^{-1}}_{\text{HS}}^2 = \sum_{n\in\mathbb Z \setminus\{0\}}\norm{A^{-1}\omega_n}^2 \leq \left(\frac{|\alpha|a(\pi-a)}{\pi}\right)^2 + \frac{\pi^2}{3}.
\end{equation*}

Realizing that $a(\pi-a)\leq \pi^2/4$, we can also obtain a bound independent of $a$.
\begin{theorem}\label{T: HS norm}
    Let $a \in (0, \pi)$, $\alpha \in \mathbb C$. The inverse $A^{-1}(a, \alpha)$ is a Hilbert-Schmidt operator with norm satisfying
\begin{equation}
    \norm{A^{-1}}_{\text{HS}}^2 \leq \left(\frac{|\alpha|a(\pi-a)}{\pi}\right)^2 + \frac{\pi^2}{3} \leq \pi^2\left(\frac{|\alpha|^2}{16} + \frac{1}{3}\right).
    \label{HSnorm}
\end{equation}
The first inequality becomes equality if $\alpha \in \mathbb R$, the second one for $a=\pi/2$.
\end{theorem}
It follows that the inverse is compact and $A$ has discrete spectrum.

\subsection{The characteristic function}
We are now fully equipped to find the characteristic function of $A$ and show how it determines the spectrum -- this constitutes the proof of Theorem \ref{CharacteristicFunction}. 

Recall that the index of an eigenvalue $\lambda \in \sigma(A)$ is defined as
\[
\iota(\lambda) \coloneqq \min \left\{k \in \mathbb N \mid \ker (A-\lambda I)^k = \ker (A-\lambda I)^{k+1}\right\}.
\]
For convenience, we set $\iota(\lambda) = 0$ whenever $\lambda \notin \sigma_p(A)$.  Notice that if, in particular $\iota(\lambda) = 1$, then the algebraic multiplicity of $\lambda$ coincides with its geometric multiplicity. Conversely, if $\lambda$ is a geometrically simple eigenvalue, then $\iota(\lambda)$ is equal to its algebraic multiplicity.

\begin{proposition}\label{IndexPole}
    Let $A$ be a densely defined operator in a Banach space that possesses a compact inverse and $\lambda \in \mathbb C$. Then the index $\iota(\lambda)$ is equal to the order of $\lambda$ as a pole of the resolvent.
\end{proposition}
\begin{proof}
    Denote $T \coloneqq A^{-1}$ the compact inverse. Then $\lambda$ is an eigenvalue of $A$ if and only if $1/\lambda$ is an eigenvalue of $T$. Moreover, both geometric and algebraic multiplicities are the same and so are the indices. \cite[Theorem 9.2.3]  {EdmundsEvans}

    Assume arbitrary $\mu \in \mathbb C \setminus \{0\}$ and denote
    \[
    P_\mu \coloneqq -\frac{1}{2\pi \mathrm{i}}\oint_\gamma(T-zI)^{-1}\dd z.
    \]
    If $\mu$ is an eigenvalue of $T$, then $P_\mu$ is the Riesz projection onto the root subspace associated with $\mu$. Here $\gamma$ is an arbitrary positively oriented Jordan curve in the resolvent set $\rho(T)$ such that no eigenvalue lies in its interior $D$ except possibly $\mu$. Then its index is 
    \[
    \iota(\mu) = \min\left\{k \in \mathbb N \mid (T-\mu I)^k P_\mu = 0\right\} < \infty.
    \]
    Thanks to the functional calculus, we have
    \[
    (T-\mu I)^kP_\mu = - \frac{1}{2\pi \mathrm{i}}\oint_\gamma (z-\mu)^k(T-zI)^{-1} \dd z, \quad k \in \mathbb N_0.
    \]
    Therefore, $\iota(\mu)$ is the lowest $k \in \mathbb N_0$ such that $(z-\mu)^k(T-zI)^{-1}$ is holomorphic in $D$. In other words, $\iota(\mu)$ is equal to the order of the pole $\mu$ of $(T-zI)^{-1}$.

    It remains to be shown that the order of the pole $1/\lambda$ of $(T-zI)^{-1}$ is precisely the order of the pole $\lambda$ of $(A-zI)^{-1}$. This follows from the simple observation that
    \[
    (A-zI)^{-1} = -\frac{1}{z}T\left(T-\frac{1}{z} I\right)^{-1}, \quad z \in \rho(A) \setminus \{0\}. \qedhere
    \]

\end{proof}

\begin{corollary}
    Let $A$ be a densely defined operator in a Banach space that possesses a compact inverse. Let $\lambda$ be a geometrically simple eigenvalue of $A$. Then the algebraic multiplicity of $\lambda$ is equal to its order as a pole of the resolvent.
\end{corollary}


Note that by \cite{BRT} all eigenvalues of $A$ are geometrically simple. As we have already found the resolvent in Proposition~\ref{Resolvent}, we can see that the order of its pole $\lambda$ is the order of $\lambda$ as a pole of the Green operator $G_\lambda$ and, as a consequence, of the Green function $\mathcal G_\lambda$. The only possible singularity of $\mathcal G_\lambda$ appears in the denominator $u_1(\pi).$ Solving \eqref{Resolventu} with zero right-hand side, $u_1(0) = 0$, and $u_1'(0) = 1$, one obtains the solution
\[
u_1(x) = \frac{1}{\lambda}\begin{cases}
\sinh(\lambda x), & 0 \leq x \leq a,\\
-\left(\cosh(\lambda\pi) +\alpha \sinh(\lambda a)\cosh(\lambda(\pi-a))\right)\sinh(\lambda(\pi-x)) \\+ \left(\sinh(\lambda\pi) + \alpha \sinh(\lambda a)\sinh(\lambda(\pi-a))\right)\cosh(\lambda(\pi-x)), & a \leq x \leq \pi.
\end{cases}
\]
It follows that the order of the pole of $(A-z I)^{-1}$ in $\lambda$ is exactly its multiplicity as a root of
\begin{equation}
S(\lambda; a, \alpha) \coloneqq u_1(\pi) = \frac{1}{\lambda}\left(\sinh(\lambda\pi) + \alpha\sinh(\lambda a) \sinh(\lambda(\pi-a))\right).
\end{equation}
Theorem~\ref{CharacteristicFunction} is thus proven.

Differentiating the function twice with respect to $\lambda$, it is a simple matter to show that no root of $S(\cdot; a, \alpha)$ is of higher multiplicity than 2.

\begin{proposition}\label{P: Roots at most double}
    All roots of $S(\cdot; a, \alpha)$ are at most double.
\end{proposition}
\begin{proof}
    Let $F(\lambda) \coloneqq \lambda S(\lambda; a, \alpha)$ for simplicity. Note that
    \[
    F(\lambda) = \sinh(\lambda\pi) + \frac{\alpha}{2}\cosh(\lambda\pi)-\frac{\alpha}{2}\cosh(\lambda(\pi-2a)).
    \]

    We have
    \begin{equation}
    F'(\lambda) = \pi\cosh(\lambda\pi) + \pi\frac{\alpha}{2}\sinh(\lambda\pi) - (\pi-2a)\frac{\alpha}{2}\sinh(\lambda(\pi-2a))
    \label{FDerivative}
    \end{equation}
    and for the second derivative
    \begin{align*}
    F''(\lambda) &= \pi^2\sinh(\lambda\pi) + \pi^2 \frac{\alpha}{2}\cosh(\lambda\pi) -(\pi-2a)^2\frac{\alpha}{2}\cosh(\lambda(\pi-2a)) \\
    &= \pi^2F(\lambda) + 2a\alpha(\pi-a)\cosh(\lambda(\pi-2a)).
    \end{align*}

    Suppose for a contradiction that 
    \begin{equation}
    F(\lambda_0) = F'(\lambda_0) = F''(\lambda_0) = 0
    \label{eq: Triple root}
    \end{equation}
    for some $\lambda_0 \in \mathbb C$. Let us first treat the case $\alpha = 0$. Then $F(\lambda_0) = \sinh(\lambda_0\pi) = 0$ means that $F'(\lambda_0) = \pi\cosh(\lambda_0\pi) \neq 0$ which contradicts \eqref{eq: Triple root}.

    If $\alpha \neq 0$, then \eqref{eq: Triple root} forces 
    \[
    \cosh(\lambda_0(\pi-2a)) = 0 \implies \lambda_0 = \frac{\mathrm{i}\pi\left(n + \frac{1}{2}\right)}{\pi - 2a}, \quad n \in \mathbb Z.
    \]
    The implication is only valid if $a \neq \pi/2$. In the case $a=\pi/2$, the left-hand side becomes $1=0$, which is a contradiction.

    Simultaneously,
    \begin{equation}
    0 = F(\lambda_0) = \sinh(\lambda_0\pi) + \frac{\alpha}{2}\cosh(\lambda_0\pi) = \sinh(\mathrm{i}\omega) + \frac{\alpha}{2}\cosh(\mathrm{i}\omega) = \mathrm{i}\sin\omega + \frac{\alpha}{2}\cos\omega,
    \label{DoubleEigenvaluesProof}
    \end{equation}
    where we denoted $\omega \coloneqq \frac{\pi^2\left(n + \frac{1}{2}\right)}{\pi-2a}$.

    If $\Re \alpha \neq 0$, then $\cos\omega = 0$ and consequently also $\sin\omega = 0$ -- a contradiction.
    
    If $\alpha \in \mathrm{i}\mathbb R$, the equation \eqref{FDerivative} for $F'(\lambda) = 0$ yields $\cos\omega = \cosh(\lambda_0\pi) = 0$ by taking the real part. However, \eqref{DoubleEigenvaluesProof} then forces also $\sin\omega = 0$ giving us the same contradiction as above.
\end{proof}

Concerning eigenvalue degeneracy, in \cite{BRT}, the authors show that every eigenvalue of $A(a, \alpha)$ with $\alpha > 0$ is geometrically simple. Although the proof for complex $\alpha$ is identical, we present it here for the convenience of the reader.

\begin{lemma}
    All eigenvalues of $A(a, \alpha)$ are non-degenerate.
\end{lemma}
\begin{proof}
    Let $\psi, \phi \in \mathcal H$ be two eigenvectors corresponding to the eigenvalue $\lambda \in \mathbb C$, that is,
    \[
    (A-\lambda I)\psi = (A-\lambda I)\phi = 0.
    \]
    Then we get
    \[
    \psi_1'' - \lambda(\lambda + \alpha \delta_a)\psi_1 = \phi_1''-\lambda(\lambda + \alpha\delta_a)\phi_1 = 0, \quad \psi_2 = \lambda \psi_1,\,\,\phi_2 = \lambda\phi_1.
    \]
    
    Assume for contradiction that $\psi$ and $\phi$ are linearly independent. Then one can choose a non-trivial linear combination $\eta = (u,v)^T \in \mathcal H$ of $\psi$ and $\phi$ such that $u'(0) = 0$. At the same time, we have $u(0) = 0$ and $(u,v)^T$ satisfy
    \[
    u'' - \lambda(\lambda + \alpha\delta_a)u = 0, \quad v = \lambda u.
    \]
    Therefore, on $(0, a)$, we have $u = v = 0$. It follows that $u'(a-) = 0$ and $v(a) = 0$ by continuity. Therefore, one also has that $u'(a+) = u'(a-) + \alpha v(a) = 0$. Consequently, $u = 0$ on $(0, \pi)$ and the same goes for $v = \lambda u$. We have found that $\eta = 0$ which contradicts its construction. Hence, $\psi$ and $\phi$ must be linearly dependent.
\end{proof}

The two results above can be summarised as follows.

\begin{corollary}\label{C: Algebraic multiplicity}
All eigenvalues of $A(a, \alpha)$ are algebraically at most double and geometrically simple. 
\end{corollary}






\section{Basis of root vectors}\label{sec.Basis}

\subsection{Rational placement of the damping}
First, we will discuss the special case $a=p\pi/q$, where $p$ and $q$ are coprime positive integers. In \cite{CH}, Cox and Henrot noticed that the characteristic function $S$ can be rewritten as
\begin{equation}
S(\lambda; p\pi/q, \alpha) = -\frac{1}{4\lambda}\mathrm{e}^{\lambda\pi}P_\alpha(\mathrm{e}^{-2\lambda\pi/q}),
\label{SusingP}
\end{equation}
where
\begin{equation}
P_\alpha(z) \coloneqq (2-\alpha)z^q + \alpha z^p + \alpha z^{q-p} - (2+\alpha).
\label{Polynomial}
\end{equation}
Note that $1$ is a simple root of $P_\alpha$ and that $\lambda \in \mathbb C$ is a root of $S(\cdot; a, \alpha)$ of multiplicity $m$ if and only if $\mathrm{e}^{-2\lambda\pi/q}$ is a root of $P_\alpha$ of multiplicity $m$.

From here it can be observed more closely why the model shows unpredictable behaviour at $\alpha = \pm 2$. Denoting $\zeta_k = |\zeta_k|\mathrm{e}^{\mathrm{i}\theta_k}$ the roots of $P_\alpha$, with the convention $\zeta_1 = 1$, we arrive at the system of eigenvalues
\begin{align}
    \lambda_{1,n}& = \mathrm{i}qn, \quad n\in \mathbb Z \setminus\{0\}, \nonumber\\
    \lambda_{k,n} &= -\frac{q}{2\pi}(\ln|\zeta_k| + \mathrm{i}(\theta_k + 2\pi n)), \quad n \in \mathbb Z,\, k\in \{2, \dots, \deg P_\alpha\}.
    \label{EigenvaluesUsingRoots}
\end{align}
If we let $\psi_{k,n}$ denote the corresponding eigenvectors of $A(p\pi/q, \alpha)$, we find
\begin{align*}
\psi_{1,n} &= \sin(nqx)\begin{pmatrix}
    1 \\ \mathrm{i}nq
\end{pmatrix}, \quad n\in \mathbb Z \setminus \{0\}, \\
\psi_{k,n} &= \begin{pmatrix}
    u_{k,n} \\ \lambda_{k,n}u_{k,n}
\end{pmatrix}, \quad n \in \mathbb Z,\, k \in \{2, \dots, \deg P_\alpha\},
\end{align*}
where
\[
u_{k,n}(x) \coloneqq \begin{cases}
    \sinh(\lambda_{k,n}(\pi-a))\sinh(\lambda_{k,n}x), & \text{for } 0 \leq x \leq a, \\
    \sinh(\lambda_{k,n}a)\sinh(\lambda_{k,n}(\pi-x)), & \text{for } a \leq x \leq \pi.
\end{cases}
\]

Similarly, if $\lambda_{k+1, n} = \lambda_{k,n}$, that is, $\zeta_k = \mathrm{e}^{-2\lambda_{k,n}\pi/q}$ is a double root of $P_\alpha$, we obtain the generalized eigenvector by solving the equation $(A-\lambda_{k,n}I)\tilde \psi_{k,n} = \psi_{k,n}$ as
\[
\psi_{k+1,n} \equiv \tilde \psi_{k,n} = \begin{pmatrix}
    \tilde u_{k,n}\\
    u_{k,n} + \lambda_{k,n} \tilde u_{k,n}
\end{pmatrix},
\]
where
\[ \tilde u_{k,n} \!=\!
\begin{cases}
    x\sinh(\lambda_{k,n}(\pi-a))\cosh(\lambda_{k,n}x) + (\pi-a)\sinh(\lambda_{k,n}x), & \!0 \leq x \leq a, \\
    -x\cosh(\lambda_{k,n}a)\cosh(\lambda_{k,n}(\pi-x)) + a\cosh(\lambda_{k,n}a)\sinh(\lambda_{k,n}(\pi-x)), &  \!a \leq x \leq \pi. 
\end{cases}
\]



By Proposition \ref{AdjointProposition}, $A^\ast(a, \alpha) = - A(a, -\overline \alpha)$; therefore, the spectrum is $\sigma(A^\ast) = \sigma_p(A^\ast) = \{\overline\lambda_{k,n}\}_{k,n}$ and the corresponding eigenvectors $\phi_{k,n}$ can be found in a similar fashion.

Combined with proper normalization, we have found a biorthogonal sequence to the sequence $\{\psi_{k,n}\}_{k,n}$ of eigenvectors of $A$. In \cite{CH}, it is shown that when normalized as $\psi_{k,n}/\lambda_{k,n}$, the sequence is Bessel, i.e. for all $\psi \in \mathcal H$ holds
\[
\sum_{k,n}\left|\inner{\psi}{\psi_{k,n}/\lambda_{k,n}}\right|^2 < +\infty.
\]
The same can be done for the biorthogonal sequence. To decide whether the systems are complete, Cox and Henrot invoked the Livšic criterion:

\begin{theorem} (Livšic, \cite[Theorem V.2.1]{GohbergKrein}).\label{Livšic}
    Let $T$ be a compact operator in a Hilbert space $\mathcal H$. Suppose that $\Re T \coloneqq \frac{1}{2}(T+T^\ast)$ is dissipative and trace-class. Then
    \[
    \Tr(\Re T) \leq \sum_{\lambda\in \sigma_p(T)}\Re  \lambda
    \]
    with eigenvalues repeated according to their algebraic multiplicity. Equality holds if and only if the root vectors of $T$ are complete in $\mathcal H$.
\end{theorem}

An analogous statement holds for accretive operators with the opposite inequality.

The dissipativity of $A(a, \alpha)$ for $\alpha > 0$ was established in \cite{BRT}. Let us extend the result for a general $\alpha \in \mathbb C.$

\begin{proposition}
    The operator $A(a,\alpha)$ is
    \begin{enumerate}
        \item maximal dissipative if and only if $\Re \alpha \geq 0$,
        \item maximal accretive if and only if $\Re \alpha \leq 0$,
        \item skew-adjoint if and only if $\Re \alpha = 0$.
    \end{enumerate}
\end{proposition}
\begin{proof}
    Let $\psi \in \Dom A$. Integrating by parts, we have
    \[
    \inner{\psi}{A\psi} = \inner{\psi_1'}{\psi_2'}_{L^2} - \inner{\psi_2'}{\psi_1'}_{L^2} - \alpha|\psi_2(a)|^2.
    \]
    Therefore,
    \[
    \Re\inner{\psi}{A\psi} = -\Re\alpha|\psi_2(a)|^2.
    \]
    The claim for dissipativity and accretivity follows. 
    
    The maximality is a consequence of Theorem \ref{CharacteristicFunction}. Clearly, for $\Re \alpha \geq 0$ and $\lambda > 0$, we have 
    \[
    \Re S(\lambda; a, \alpha) > 0.
    \]
    Therefore, $(0, +\infty) \subset \rho(A(a,\alpha))$. We proceed analogously for $\Re \alpha \leq 0$ to show that $(-\infty, 0) \subset \rho(A(a, \alpha))$.
    
    The characterisation of skew-adjointness is an immediate consequence of Proposition \ref{AdjointProposition}.
\end{proof}

It remains to apply the criterion to the compact inverse $T \coloneqq A^{-1}$ whose root vectors coincide with those of $A$. With $\Re \alpha \geq 0$ (respectively, $\Re \alpha \leq 0$), it is clearly dissipative (respectively, accretive) as an inverse of a dissipative (respectively, accretive) operator. Therefore, its real part is also dissipative (respectively, accretive) due to the identity
\[
\Re \inner{\psi}{\Re T\psi} = \Re \inner{\psi}{T\psi}.
\]

As $\Re A^{-1}$ is a one-dimensional (and thus trace-class) operator, calculation of the trace is simple and for $\alpha \in \mathbb R$ done in \cite{CH}. The extension for all complex $\alpha$ is straightforward but for self-containment, we provide the proof:
\begin{proposition}\label{Trace}
$\Tr\left(\Re A^{-1}(a, \alpha)\right) = - \frac{\Re \alpha(\pi-a)a}{\pi}$.
\end{proposition}
\begin{proof}
    It follows by Proposition \ref{AdjointProposition} and the form of the resolvent \eqref{Resolvent} that
    \[
    \Re A^{-1} = \frac{1}{2}(A^{-1} + (A^{-1})^\ast) = \begin{pmatrix}
        G_0(\Re \alpha \delta_a) & 0 \\
        0 & 0
    \end{pmatrix}.
    \]
    Therefore, using the definitions \eqref{GreenOperator} and \eqref{eq: Green function}, we arrive at
    \[
    \Re A^{-1}\begin{pmatrix}
        u \\ v
    \end{pmatrix} = \Re(\alpha)u(a)\begin{pmatrix}
       \mathcal G_0(\cdot, a) \\ 0
    \end{pmatrix}, \quad \begin{pmatrix}
        u \\ v
    \end{pmatrix} \in \mathcal H.
    \]
    Let us denote
    \[
    \rho(x) \coloneqq \pi \mathcal G_0(x, a) = \begin{cases}
        (\pi - a)x, &0 \leq x \leq a, \\
        a(\pi-x), &a \leq x \leq \pi.
    \end{cases}
    \]
    We can see that $\ran \Re A^{-1}$ is the one-dimensional subspace generated by $(\rho, 0)^T$.

    For the calculation of the trace, it is advantageous to choose any orthonormal basis that contains the normed vector $(\rho, 0)^T$. As $\norm{\rho}^{2} = \pi a(\pi-a)$, we calculate the trace as
    \[
    \Tr \Re A^{-1}  = \frac{1}{\pi a(\pi-a)}\inner{\begin{pmatrix}
    \rho \\ 0 
\end{pmatrix}}{\Re A^{-1}\begin{pmatrix}
    \rho \\ 0 
\end{pmatrix}} = -\frac{\Re \alpha}{\pi^2}\norm{\rho'}^2 = -\frac{\Re \alpha(\pi-a)a}{\pi}. \qedhere
\]
    
\end{proof}

We note that Proposition \ref{Trace} and Theorem \ref{Livšic} ensure that the spectral sum
\[
\sum_{\lambda \in \sigma_p(A^{-1})}\Re \lambda = \sum_{\lambda \in \sigma_p(A)} \Re \frac{1}{\lambda}
\]
converges. In fact, due to the dissipativity (respectively, accretivity) of $A$, the summands are all of the same sign, so the convergence is monotone.

For $\alpha > 0$, $\alpha\neq 2$, Cox and Henrot found that also
\[
\sum_{\lambda\in\sigma(A)} \Re \frac{1}{\lambda} = -\frac{\alpha(\pi-a)a}{\pi}
\]
and thus arrived at the result that the root vectors are complete by Theorem \ref{Livšic}. Combined with the Bessel property and the same two properties of the biorthogonal sequence, the root vectors form a Riesz basis in $\mathcal H$ due to Theorem \ref{RieszCriterion}.

Here, we will provide a detailed calculation for $\alpha = 2$ and extend the result for $\alpha = -2$. Recall that for $\alpha = 2$, the degree of the polynomial \eqref{Polynomial} is $r \coloneqq \max\{p, q-p\} < q$. Also note that, as pointed out in \cite{CH}, simply by differentiating the characteristic function \eqref{DefinitionOfS} from definition, we obtain the Taylor expansion
\begin{equation}
S(\lambda; a, \alpha) = \pi + \alpha a(\pi-a)\lambda + \mathcal O(\lambda^2).
    \label{Taylor}
\end{equation}

On the other hand, the following equality holds:

\begin{proposition}\label{TaylorWithP}
    $S(\lambda; p\pi/q, 2) = \pi - \pi\lambda\left(\sum_{\lambda\in\sigma(A)}\Re \frac{1}{\lambda} - \frac{\pi(q-r)}{q}\right) + \mathcal O(\lambda^2)$.
\end{proposition}
\begin{proof}
    Note that $P_2(z) = 2(z^r + z^{q-r}-2)$. Then using \eqref{SusingP}, we have
    \[
    \lambda S(\lambda; p\pi/q, 2) \equiv F(\lambda) = -\frac{1}{4}\mathrm{e}^{\lambda\pi}P_{2}\left(\mathrm{e}^{-2\lambda\pi/q}\right) = -\frac{1}{2}\mathrm{e}^{\lambda\pi}\prod_{k=1}^r \left(\mathrm{e}^{-2\lambda\pi/q}-\zeta_k\right),
    \]
    where $\zeta_k$ are the roots of $P_2$, $\zeta_1 = 1$. Differentiating $F$, we have
    \[
    F'(\lambda) = - \frac{\pi}{2}\mathrm{e}^{\lambda\pi}\prod_{k=1}^r\left(\mathrm{e}^{-2\lambda\pi/q} - \zeta_k\right) + \frac{\pi}{q}\mathrm{e}^{\lambda\pi(1-2/q)}\sum_{j=1}^r\prod_{k\neq j}\left(\mathrm{e}^{-2\lambda\pi/q} - \zeta_k\right).
    \]
    Therefore,
    \[
    F'(0) = \frac{\pi}{q}\prod_{k=2}^r\left(1-\zeta_k\right) = \pi.
    \]
    Differentiating again, we obtain
    \begin{align*}
        F''(\lambda) = &\frac{\pi^2}{2}\mathrm{e}^{\lambda\pi}\prod_{k=1}^r\left(\mathrm{e}^{-2\lambda\pi/q}-\zeta_k\right) + \frac{2\pi^2}{q^2}(q-1)\mathrm{e}^{\lambda\pi(1-2/q)}\sum_{j=1}^r\prod_{k\neq j}\left(\mathrm{e}^{-2\lambda\pi/q}-\zeta_k\right) \\
        &- \frac{2\pi^2}{q^2}\mathrm{e}^{\lambda\pi(1-4/q)}\sum_{i,j =1}^r\prod_{k\neq i,j}\left(\mathrm{e}^{-2\lambda\pi/q}-\zeta_k\right).
    \end{align*}
    At $\lambda = 0$ this becomes
    \begin{align}
    F''(0) &= \frac{2\pi^2(q-1)}{q^2}\prod_{k=2}^r\left(1-\zeta_k\right) - \frac{2\pi^2}{q^2}\sum_{j=2}^r\prod_{k\neq 1, j}(1-\zeta_k) = \frac{2\pi^2(q-1)}{q}-\frac{4\pi^2}{q}\sum_{k=2}^r\frac{1}{1-\zeta_k} \nonumber\\
    & = \frac{2\pi^2}{q}\left[\sum_{k=2}^r \frac{\zeta_k+1}{\zeta_k - 1} + q-r\right].
    \label{Computation}
    \end{align}

    At the same time, \eqref{EigenvaluesUsingRoots} gives us
    \[
    \frac{1}{\lambda_{k,n}} = -\frac{2\pi}{q}\frac{\ln|\zeta_k| -\mathrm{i}(\theta_k + 2\pi n)}{\ln^2|\zeta_k| + (\theta_k + 2\pi n)^2}
    \]
    for $k \in \{1, \dots, \deg P_\alpha\}$ and $n \in \mathbb Z$. Therefore, the real parts are
    \[
    \Re \frac{1}{\lambda_{k,n}} = -\frac{\ln |\zeta_k|}{2\pi q}\frac{1}{\frac{\ln^2|\zeta_k|}{4\pi^2} + \left(n+\frac{\theta_k}{2\pi}\right)^2}.
    \]
    Note the for the sum of real parts, we may ignore the imaginary eigenvalues $\lambda_{1,n}$.

    To proceed with the proof, we need to know the sum of the following series.
    \begin{lemma}
        Let $\beta, \,\gamma \in \mathbb R$, $\beta \neq 0$. Then
        \begin{equation}
            \sum_{n\in\mathbb Z}\frac{1}{(n+\gamma)^2 + \beta^2} = \frac{\pi}{2\beta}\frac{\sinh(2\pi\beta)}{\cosh^2(\pi\beta) - \cos^2(\pi\gamma)}.
            \label{TheSeries}
        \end{equation}
    \end{lemma}
\begin{proof}
    It is simple to verify the assumptions for Poisson summation \cite[Theorem 2.4]{PoissonSum}. The Fourier transform of $f(x) = \frac{1}{(x+\gamma)^2 + \beta^2}$ is
    \[
    \hat f(\xi) = \int_{\mathbb R} \frac{\mathrm{e}^{-2\pi \mathrm{i}\xi x}}{(x+\gamma)^2 + \beta^2} \dd x = \mathrm{e}^{2\pi \mathrm{i}\xi \gamma}\int_{\mathbb R}\frac{\mathrm{e}^{-2\pi \mathrm{i}\xi x}}{x^2 + \beta^2} \dd x = \frac{\pi}{\beta}\mathrm{e}^{2\pi \mathrm{i}\xi \gamma - 2 \pi \mathrm{i} |\xi| \beta}
    \]
    using residues. Poisson summation then yields
    \begin{align*}
    \sum_{n\in \mathbb Z}f(n) &= \sum_{n\in\mathbb Z}\hat f(n) = \frac{\pi}{\beta}\sum_{n\in\mathbb Z}\mathrm{e}^{2 \pi \mathrm{i}(n\gamma - |n|\beta)} = \frac{\pi}{\beta}\left(\frac{1}{1-\mathrm{e}^{2\pi \mathrm{i}\gamma-2\pi\beta}} + \frac{1}{1-\mathrm{e}^{-2\pi\mathrm{i}\gamma-2\pi\beta}} - 1\right) \\
    &= \frac{\pi}{2\beta}\frac{\sinh(2\pi\beta)}{\cosh^2(\pi\beta) - \cos^2(\pi\gamma)} \qedhere
    \end{align*}
\end{proof}

Let us get back to the proof of Proposition \ref{TaylorWithP}. Setting $\gamma \coloneqq \theta_k/2\pi$ and $\beta \coloneqq \ln|\zeta_k|/2\pi$ in \eqref{TheSeries}, it follows that
\begin{align}
    \sum_{n\in\mathbb Z}\Re \frac{1}{\lambda_{k,n}} & = -\frac{\ln|\zeta_k|}{2\pi q}\frac{\pi^2}{\ln|\zeta_k|}\frac{\sinh\left(\ln|\zeta_k|\right)}{\cosh^2\left(\frac{\ln|\zeta_k|}{2}\right)-\cos^2(\theta_k/2)} = \frac{\pi}{q}\frac{1-|\zeta_k|^2}{|\zeta_k|^2 - 2\Re \zeta_k + 1} \nonumber\\
    & =\frac{\pi}{q}\Re\frac{\zeta_k + 1}{1- \zeta_k}.
    \label{SpectralTraceComputation}
\end{align}

Comparing with \eqref{Computation}, we have
\[
\Re F''(0) = -2\pi \sum_{k=2}^r\sum_{n\in \mathbb Z}\Re \frac{1}{\lambda_{k,n}} + \frac{2\pi^2(q-r)}{q}.
\]
This completes the proof since $F''(0)$ is real and all eigenvalues are repeated according to their algebraic multiplicity. Indeed, all eigenvalues are geometrically simple and thus the algebraic multiplicity coincides with the index, which in turn equals the eigenvalue's root multiplicity in $S(\cdot; a, \alpha)$. We have already noted that this is the same as the root multiplicity of $\mathrm{e}^{-2\lambda\pi/q}$ in $P_\alpha$.
\end{proof}

Comparing Proposition \ref{TaylorWithP} and \eqref{Taylor}, we can conclude with the following
\begin{theorem}\label{SpectralTrace}
    Let $a = p\pi/q$, $\alpha= \pm 2$. Then
    \[
    \sum_{\lambda \in \sigma(A)}\Re \frac{1}{\lambda} = - \frac{\alpha(\pi-a)a}{\pi} \pm \frac{\pi}{q}(q-r) = \begin{cases}
        -\frac{\alpha(\pi-a)a}{\pi} \pm a, & \text{for } 0 < a \leq \pi/2, \\
        -\frac{\alpha(\pi-a)a}{\pi} \pm (\pi-a), & \text{for } \pi/2 \leq a < \pi.
    \end{cases}
    \]
\end{theorem}
\begin{proof}
    For $\alpha = -2$, Proposition \ref{AdjointProposition} gives $\Re A^{-1}(a, -2) = -\Re A^{-1}(a, 2)$.
\end{proof}

Using Proposition \ref{Trace}, we can see that for $\alpha = 2$ holds
\[
\Tr\left(\Re A^{-1}\right) < \sum_{\lambda \in \sigma (A)} \Re \frac{1}{\lambda}.
\]
Analogously, for $\alpha = -2$ we have
\[
\Tr\left(\Re A^{-1}\right) > \sum_{\lambda \in \sigma (A)} \Re \frac{1}{\lambda}.
\]
Illustrations of both quantities are provided by Figure \ref{fig:Traces}.

\begin{figure}[ht]
    \centering
    \begin{subfigure}{0.45\textwidth}
        \includegraphics[width=\linewidth]{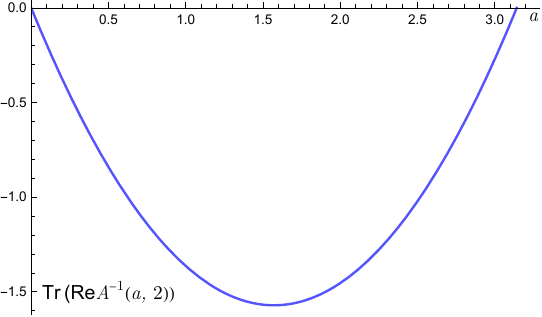}
        \caption{The trace $\Tr(\Re A^{-1}(a,\alpha))$ for $\alpha=2$.}
    \end{subfigure}
    \hfill
    \begin{subfigure}{0.45\textwidth}
        \includegraphics[width=\linewidth]{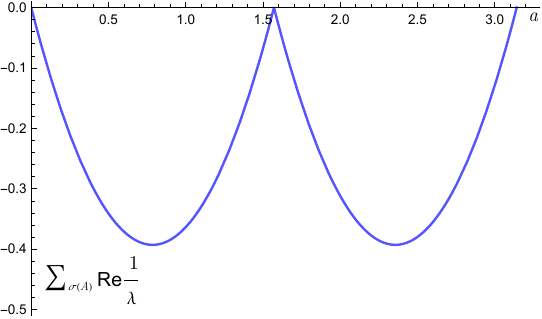}
        \caption{The sum $\sum_{\sigma(A)} \Re \frac{1}{\lambda}$ for $\alpha=2$.}
    \end{subfigure}
    \medskip 
    
    \begin{subfigure}{0.45\textwidth}
        \includegraphics[width=\linewidth]{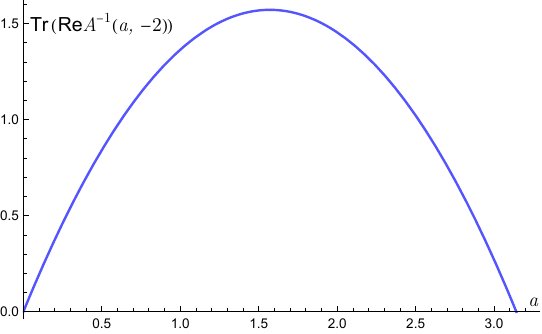}
        \caption{The trace $\Tr(\Re A^{-1}(a,\alpha))$ for $\alpha=-2$.}
    \end{subfigure}
    \hfill    
    \begin{subfigure}{0.45\textwidth}
        \includegraphics[width=\linewidth]{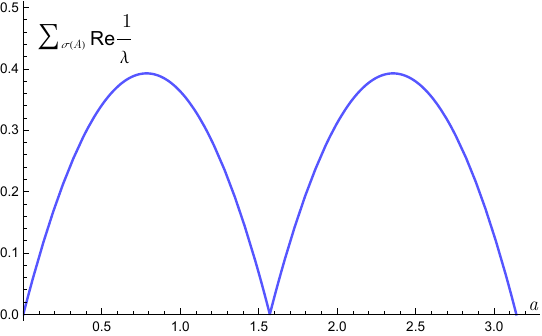}
        \caption{The sum $\sum_{\sigma(A)} \Re \frac{1}{\lambda}$ for $\alpha=-2$.}
    \end{subfigure}
     \caption{An illustration of the left- and right-hand sides of the condition in the Livšic criterion (Theorem \ref{Livšic}) as functions of $a$ with $\alpha= \pm 2$.}
    \label{fig:Traces}
\end{figure}

Thanks to the Livšic criterion (Theorem \ref{Livšic}), we can state the desired result for $\alpha = \pm 2$:

\begin{theorem}\label{RieszBasisCriterion}
    Let $a= p\pi/q$ and $\alpha= \pm 2$. Then the root vectors of $A(a, \alpha)$ are not complete in $\mathcal H$. Consequently, they do not form a Riesz basis.
\end{theorem}

\subsection{Complex damping parameter}
By careful analysis of the proof of Proposition \ref{TaylorWithP}, we are able to determine the Riesz basis property also for an arbitrary $\alpha \in \mathbb C \setminus \{\pm2\}$.

\begin{proposition}
    Let $\alpha \in \mathbb C \setminus\{\pm 2\}.$ Then 
    $
    \Re S'(0; p\pi/q, \alpha) = -2\pi\sum_{\lambda\in\sigma(A)}\Re \frac{1}{\lambda}.
    $
\end{proposition}
\begin{proof}
    Note that for $\alpha \neq 2$, the polynomial $P_\alpha$ is of degree $q$. As a consequence, in \eqref{Computation}, we obtain
    \begin{equation}
    F''(0) = \frac{2\pi^2}{q}\sum_{k=2}^q\frac{\zeta_k+1}{\zeta_k-1}.
    \label{FSeconedDerivative}
    \end{equation}

    Since the calculation \eqref{SpectralTraceComputation} does not depend on $\alpha$, we can compare it with \eqref{FSeconedDerivative} to conclude
    \[
    \Re F''(0) = -2\pi \sum_{k=2}^q\sum_{n \in \mathbb Z}\Re \frac{1}{\lambda_{k,n}}. \qedhere
    \]
\end{proof}

Recalling the Taylor expansion \eqref{Taylor}, the above proposition yields for $a = p\pi/q$ and any $\alpha \in \mathbb C \setminus \{\pm2\}$:
\begin{equation}
\sum_{\lambda\in\sigma(A)}\Re\frac{1}{\lambda} = -\frac{\Re\alpha(\pi-a)a}{\pi} = \Tr\left(\Re A^{-1}(a, \alpha)\right),
\label{SpectralTraceGeneral}
\end{equation}
recalling Proposition \ref{Trace}.
Livšic criterion \ref{Livšic} then ensures that the root vectors are complete. The Bessel property is independent of $\alpha$. Therefore, using Theorem \ref{RieszCriterion}, we have managed to generalize the result of \cite{CH} also for $\alpha \in \mathbb C$:

\begin{theorem}
    Let $\alpha \in \mathbb C \setminus\{\pm 2\}$, $a = p\pi/q$. Then the root vectors of $A(p\pi/q, \alpha)$ form a Riesz basis in $\mathcal H$.
\end{theorem}

\subsection{General placement of the damping}
In this section, we aim to extend the result to an arbitrary placement of the damping $a \in (0, \pi)$. We first note that if $a \notin \pi\mathbb Q$, then no purely imaginary eigenvalues exist except for the skew-adjoint case $\alpha \in \mathrm{i}\mathbb R$.

\begin{lemma}\label{L: No harmonic modes}
    Let $a \in (0,\pi)$ with $a \notin \pi\mathbb Q$ and $\alpha \in \mathbb C \setminus \mathrm{i}\mathbb R$. Then $\sigma(A) \cap \mathrm{i}\mathbb R = \emptyset$.
\end{lemma}
\begin{proof}
   Assume for contradiction that $S(\lambda;a,\alpha) = 0$ with $\lambda = \mathrm{i}\mu$, where $\mu \in \mathbb R \setminus \{0\}$. Then we find
   \[
   \mathrm{i}\sin(\mu\pi) - \alpha\sin(\mu a)\sin(\mu(\pi-a)) = 0.
   \]
   As $\Re \alpha \neq 0$, taking the real part yields $\sin(\mu a)\sin(\mu(\pi-a)) = 0$. Using this knowledge then gives also $\sin(\mu \pi) = 0$. Combined, the identities yield 
    \[\sin(\mu \pi) = \sin(\mu a) = \sin(\mu(\pi-a)) = 0.
    \]
    This readily gives us $a \in \mathbb Q \pi$ -- a contradiction.
\end{proof}

From Theorem \ref{SpectralTrace}, Proposition \ref{Trace} and the way they are used for Theorem \ref{Livšic}, it obviously suffices to show that the sum of the series
\[
\sum_{\lambda \in \sigma(A(a, \alpha))} \Re\frac{1}{\lambda}
\]
is continuous in $a$ to prove that the root vectors are complete. We will make use of the following recent result of Krejčiřík and Lipovský:
\begin{theorem} \cite[Theorem 4.2]{KrLip}. \label{BoundednessOfEigenvalues} Let $\alpha \in \mathbb C$ be arbitrary and $a \in (0, \pi/2)$. Let $\lambda_j^+(a)$ denote the $j$-th eigenvalue in the upper half-plane sorted in non-decreasing order according to the imaginary part. Then
\[
\lambda_j^+(a) = \begin{cases}
    \mathrm{i}j + f_j(a), & \text{for } \alpha \neq \pm 2, \\
    \frac{\mathrm{ij\pi}}{\pi-a} + f_j(a), & \text{for } \alpha = \pm 2,
\end{cases}
\]
where
\begin{enumerate}
    \item $f_j$ are analytic in $a$ with at most algebraic singularities. If for certain $a_0$ a finite number of $\lambda_j(a_0)$ have the same imaginary parts, one may need to interchange their indices to get the analyticity.
    \item $|\Re f_j(a)| \leq c_1$, where $c_1 > 0$ is independent of $j$ and $a$.
    \item $|\Im f_j(a)| \leq c_2$, where $c_2 > 0$ is independent of $j$ and $a$.
\end{enumerate}
Analogous statement holds for the eigenvalues $\lambda_j^-(a)$ in the lower half-plane.    
\end{theorem}

Since 
\[
\left|\Re \frac{1}{\lambda_j^\pm(a)}\right| = \frac{\left|\Re \lambda_j^\pm(a)\right|}{|\lambda_j^\pm(a)|^2} \leq \frac{c_1}{|\lambda_j^\pm(a)|^2}
\]
and the last term is eventually dominated by $\varepsilon j^{-2}$ for some $\varepsilon >0$, the series converges uniformly in $a$. Consequently, the sum is continuous with respect to $a$ and the root vectors remain complete if and only if $\alpha \neq \pm 2$.

To admit any $a \in (0, \pi)$ we only need to realize that the trace in Proposition \ref{Trace}, Theorem \ref{SpectralTrace}, and \eqref{SpectralTraceGeneral} are symmetric under the exchange of $a$ and $\pi-a$. We have proven the following result

\begin{lemma}\label{L: Total condition}
    Let $a \in (0, \pi)$ and $\alpha \in \mathbb C$. Then the root vectors of $A(a, \alpha)$ are complete in $\mathcal H$ if and only if $\alpha \neq \pm 2$.
\end{lemma}

It is left to check the biorthogonality and the Bessel property. The situation here is significantly more involved than in the case of a rational placement. It turns out that a biorthogonal pairing with the Bessel property preserved becomes impossible for certain values of $\alpha$.

Recall that for an eigenvalue $\lambda \in \mathbb C$, the operator $A$ possesses the eigenvector
\[
    \psi_\lambda = \begin{pmatrix}
        u_\lambda \\ \lambda u_\lambda
    \end{pmatrix}, \quad u_\lambda(x) \coloneqq \begin{cases}
        \sinh(\lambda(\pi-a))\sinh(\lambda x), & \text{for } 0 \leq x \leq a, \\
        \sinh(\lambda a)\sinh(\lambda(\pi-x)), & \text{for } a \leq x \leq \pi.
    \end{cases}
    \]
Again, if $\lambda$ is a double root of $S(\cdot; a, \alpha)$, then there is also the generalized eigenvector
    \[
    \tilde \psi_\lambda = \begin{pmatrix}
        \tilde u_\lambda \\
        u_\lambda + \lambda \tilde u_\lambda
    \end{pmatrix},
    \]
    where we denoted
    \[
    \tilde u_\lambda(x) \coloneqq \begin{cases}
        x\sinh(\lambda(\pi-a))\cosh(\lambda x) + (\pi-a)\sinh(\lambda x), &\text{for } 0 \leq x \leq a, \\
        -x\cosh(\lambda a)\cosh(\lambda(\pi-x)) + a\cosh(\lambda a)\sinh(\lambda(\pi-x)),& \text{for } a \leq x \leq \pi. 
    \end{cases}
    \]
 Higher algebraic multiplicities are forbidden by Corollary \ref{C: Algebraic multiplicity}.
 Similarly, for the adjoint operator $A^\ast(a,\alpha) = -A(a, -\overline{\alpha})$ and its eigenvalue $\overline{\lambda}$, we get the eigenvector
 \[
 \phi_\lambda = \begin{pmatrix}
     \overline{u} \\ - \overline{\lambda u}
 \end{pmatrix}
 \]
 and potentially also the generalized eigenvector
 \[
 \tilde \phi_\lambda =\begin{pmatrix}
        \overline{\tilde u} \\ -\overline{\lambda\tilde u} - \overline{u}
    \end{pmatrix}.
 \]

Given $a \in (0, \pi)$ with $a \notin \pi\mathbb Q$, we define the set
\begin{equation}
\mathcal D_a \coloneqq
\left\{
(X,Y) \in (\mathbb C \setminus \{0\})^2
\,\middle|\,
\begin{array}{l}
|X|^{\pi-a} = |Y|^a,\\
aX(Y-1)^2 + (\pi-a)Y(X-1)^2 = 0,\\
(X,Y) \neq (1,1)
\end{array}
\right\}.
\label{eq: D_a definition}
\end{equation}
This set gives rise to the set of exceptional values of $\alpha$:
\begin{equation}
    \mathcal E_a \coloneqq \left\{-2\frac{XY-1}{(X-1)(Y-1)} \mid (X, Y) \in \mathcal D_a\right\}.
    \label{eq: Exceptional alpha}
\end{equation}

\begin{lemma}\label{L: Double eigenvalues are exceptional}
    Let $a \in (0, \pi)$ with $a \notin \pi \mathbb Q$ and $\alpha \in \mathbb C$. If $A(a, \alpha)$ has a double eigenvalue, then $\alpha \in \mathcal E_a$.
\end{lemma}
\begin{proof}
    Let us recall the notation 
    \[
    F(\lambda) = \lambda S(\lambda;a, \alpha) = \sinh(\lambda\pi) + \alpha\sinh(\lambda a)\sinh(\lambda(\pi-a))
    \]
    of Proposition \ref{P: Roots at most double}. If $\lambda_0 \in \mathbb C$ is a double eigenvalue of $A(a, \alpha)$, then
    \[
    F(\lambda_0) = F'(\lambda_0) = 0.
    \]
    Let us denote
    \[
    X_0 \coloneqq \mathrm{e}^{2a\lambda_0}, \quad Y_0 \coloneqq \mathrm{e}^{2(\pi-a)\lambda_0}.
    \]
    Then clearly $|X_0|^{\pi-a} = |Y_0|^a$. Note that $\sinh(\lambda_0 a)\sinh(\lambda_0 (\pi-a)) \neq 0$ by Lemma \ref{L: No harmonic modes}. This also justifies that $X_0, Y_0 \neq 1$.  Moreover, one has
    \begin{align*}
        -2\frac{X_0Y_0 - 1}{(X_0-1)(Y_0-1)} = - \frac{\sinh(\lambda_0\pi)}{\sinh(\lambda_0 a)\sinh(\lambda_0(\pi-a))} = \alpha
    \end{align*}
    due to the equation $F(\lambda_0) = 0$. Finally, a direct calculation from $F'(\lambda_0) = 0$ using the above expression for $\alpha$ yields
    \[
    aX_0(Y_0-1)^2 + (\pi-a)Y_0(X_0-1)^2 = 0.
    \]
    Therefore, $(X_0, Y_0) \in \mathcal D_a$ and $\alpha \in \mathcal E_a$.
\end{proof}


\begin{theorem}\label{T: Irrational Riesz criterion}
    Let $\alpha \in \mathbb C$ and $a \in (0,\pi)$ with $a \notin \pi \mathbb Q$. Let $\mathcal E_a$ be given by \eqref{eq: Exceptional alpha}. Then exactly one of the following items occurs:
    \begin{enumerate}
        \item If $\alpha = \pm 2$, the root vectors of $A(a,\alpha)$ are not complete in $\mathcal H$.
        \item If $\alpha \in \mathcal E_a$, then the root vectors of $A(a, \alpha)$ are complete in $\mathcal H$, but they do not form a Riesz basis.
        \item If $\alpha \notin \mathcal E_a \cup \{\pm 2\}$, then the root vectors of $A(a, \alpha)$ form a Riesz basis in $\mathcal H$.
    \end{enumerate}
\end{theorem}
\begin{proof}
    We will prove each item separately.
    \begin{enumerate}
        \item See Lemma \ref{L: Total condition}.
        \item We must first prove that $\pm 2 \notin \mathcal E_a$. Indeed, solving the system
        \[
        \alpha = -2\frac{XY - 1}{(X-1)(Y-1)}, \quad aX(Y-1)^2 + (\pi-a)Y(X-1)^2 = 0, \quad (X, Y) \neq (1,1)
        \]
        with $\alpha= \pm 2$ yields
        \[
        (X,Y) = \begin{cases}
            \left(\frac{\pi-2a}{2(\pi-a)}, -\frac{\pi-2a}{2a}\right), & \alpha = 2,\\
            \left(\frac{2(\pi-a)}{\pi-2a}, -\frac{2a}{\pi-2a}\right), & \alpha = -2.
        \end{cases}
        \]
    
        The modulus equation then becomes identical:
        \[
        \left(\frac{2(\pi-a)}{\pi-2a}\right)^{\pi-2a} = \left(\frac{2a}{\pi-2a}\right)^a.
        \]
        Without losing generality, we may assume that $a < \pi/2$ (otherwise we can exchange the roles of $a$ and $\pi - a$). Denoting $r \coloneqq (\pi-a)/a > 1$, we calculate
        \begin{align*}
            \ln\left(\frac{\left(\frac{2(\pi-a)}{\pi-2a}\right)^{\pi-2a}}{\left(\frac{2a}{\pi-2a}\right)^a}\right) = a\left(r\ln r - (r-1)\ln(r-1) + (r-1)\ln 2\right) > 0,
        \end{align*}
        because $r \mapsto r\ln r - (r-1)\ln(r-1)$ is strictly increasing for $r > 1$ and tends to $0$ as $r \searrow 1$. Therefore, we have found that 
        \[
        \left(\frac{2(\pi-a)}{\pi-2a}\right)^{\pi-2a} > \left(\frac{2a}{\pi-2a}\right)^a
        \]
        and consequently, $\alpha = \pm 2 \notin \mathcal E_a$.

        Lemma \ref{L: Total condition} again shows that the root vectors are complete.

        Now we will prove that $\alpha \in \mathcal E_a$ is indeed an obstruction for the Riesz basis. It suffices to show that the spectral projections 
        \[
        P_\lambda = - \frac{1}{2\pi \mathrm{i}}\int_\Gamma (A-zI)^{-1}\dd z
        \]
        are not uniformly bounded. Take $(X,Y) \in \mathcal D_a$ such that 
        \[
        \alpha = - 2\frac{XY-1}{(X-1)(Y-1)}.
        \]
        Then, by the definition \eqref{eq: D_a definition} of $D_a$, there exists $s \in \mathbb R$ and $\theta, \phi \in [0,2\pi)$ such that
        \[
        X = \mathrm{e}^{2as}\mathrm{e}^{\mathrm{i}\theta}, \quad Y = \mathrm{e}^{2(\pi-a)s}\mathrm{e}^{\mathrm{i}\phi}.
        \]
        Since $a/\pi \notin \mathbb Q$, it follows that $a/(\pi-a)$ is also irrational. Therefore, we can take a sequence $\{\xi_n\}_{n=1}^{\infty}$ with $\xi_n \to \infty$ such that
        \[
        \lim_{n\to\infty}\mathrm{e}^{2\mathrm{i}a\xi_n}  = \mathrm{e}^{\mathrm{i}\theta}, \quad \lim_{n\to \infty}\mathrm{e}^{2\mathrm{i}(\pi-a)\xi_n} = \mathrm{e}^{\mathrm{i}\phi}.
        \]
        Defining $\mu_n \coloneqq s + \mathrm{i}\xi_n$, we can see that
        \[
        \lim_{n\to\infty}\mathrm{e}^{2a\mu_n} = X, \quad \lim_{n\to\infty}\mathrm{e}^{2(\pi-a)\mu_n} = Y.
        \]

        Let us now multiply the characteristic function by a suitable factor and rewrite it similarly to the case of a rational placement:
        \[
        G(\lambda) \coloneqq 4\mathrm{e}^{\pi\lambda}F(\lambda) = (2+\alpha)\mathrm{e}^{2\pi\lambda} - \alpha \mathrm{e}^{2a\lambda} - \alpha \mathrm{e}^{2(\pi-a)\lambda} + \alpha - 2 
        \]
        Consider now the translated functions
        \[
        G_n(\lambda) \coloneqq G(\lambda + \mu_n).
        \]
        It is a simple matter to show that these functions converge locally uniformly to the function
        \[
        G_\infty(\lambda) \coloneqq (2+\alpha)XY\mathrm{e}^{2\pi\lambda} - \alpha X\mathrm{e}^{2a\lambda} - \alpha Y\mathrm{e}^{2(\pi-a)\lambda} + \alpha - 2.
        \]
        The definition of $D_a$ now ensures that $0$ is a double root of $G_\infty$.

        As a consequence of Rouché's theorem, $G_n$ has precisely two roots $\eta_n^{(1)}$, $\eta_n^{(2)}$ in a sufficiently small neighbourhood of $0$. These roots must satisfy
        \[
        \lim_{n\to\infty}\eta_n^{(j)} = 0, \quad j \in \{1, 2\}.
        \]
        Consequently, if we define
        \[
        \lambda_n^{(j)} \coloneqq \mu_n + \eta_n^{(j)}, \quad j \in \{1, 2\}, \quad n \in \mathbb N,
        \]
        we obtain two roots of $G$; that is, two eigenvalues of $A(a, \alpha)$. By construction, these eigenvalues satisfy
        \[
        \lambda_n^{(j)} = s + \mathrm{i}\xi_n + o(1), \quad n\to \infty, \quad j \in \{1, 2\}.
        \]

        Next, we shall prove that the eigenvalues constructed above are eventually distinct and simple. To this end, it suffices to show that with $a \notin \pi\mathbb Q$, there are only finitely many double eigenvalues. Note that $0 \notin \mathcal E_a$, so we can assume that $\alpha \neq 0$. Given any eigenvalue $\lambda$, we can define
        \[
        X_\lambda \coloneqq \mathrm{e}^{2a\lambda}, \quad Y_\lambda \coloneqq \mathrm{e}^{2(\pi-a)\lambda}.
        \]
        Equations $G(\lambda) = G'(\lambda) = 0$ then produce the quadratic equation
        \[
        (\pi-a)\alpha(\alpha + 2)X_\lambda^2 + (4\pi - 2(\pi-a)\alpha^2)X_\lambda + (\pi-a)\alpha(\alpha-2) = 0
        \]
        with $Y_\lambda$ determined by
        \[
        Y_\lambda = \frac{\pi(\alpha-2) - (\pi-a)\alpha X_\lambda}{a\alpha}.
        \]
        Therefore, there are at most two such pairs $(X_\lambda, Y_\lambda)$. That these correspond to at most two double eigenvalues $\lambda$ follows from the injectivity of the mapping $\lambda \mapsto (\mathrm{e}^{2a\lambda}, \mathrm{e}^{2(\pi-a)\lambda})$ provided by $a \notin \pi\mathbb Q$.

        Let us denote
        \[
        \psi_n^{(j)} \coloneqq \psi_{\lambda_n^{(j)}}, \quad \phi_n^{(j)} \coloneqq \phi_{\lambda_n^{(j)}}, \quad n \in \mathbb N, \quad j \in \{1, 2\}.
        \]
        By the previous paragraph, we can assume without loss of generality that the spectral projections corresponding to $\lambda_n^{(j)}$ are of the form
        \[
        P_n^{(j)} \coloneqq \frac{\inner{\phi_n^{(j)}}{\cdot}}{\inner{\phi_n^{(j)}}{\psi_n^{(j)}}}\psi_n^{(j)}.
        \]
        Their norm can then be found simply as
        \[
        \norm{P_n^{(j)}} = \frac{\norm{\phi_n^{(j)}} \norm{\psi_n^{(j)}}}{\left|\inner{\phi_n^{(j)}}{\psi_n^{(j)}}\right|}.
        \]
        A direct calculation yields
        \begin{equation}
        d_\lambda \coloneqq \frac{1}{|\lambda|^2}\inner{\phi_\lambda}{\psi_\lambda} = a\sinh^2(\lambda(\pi-a)) + (\pi-a)\sinh^2(\lambda a).
        \label{eq: d_lambda}
        \end{equation}
        Along our sequences, using $\mathrm{e}^{2a\lambda_n^{(j)}} \to X$ and $\mathrm{e}^{2(\pi-a)\lambda_n^{(j)}} \to Y$, we get
        \begin{equation}
        \lim_{n\to \infty}d_{\lambda_n^{(j)}} = a\frac{(X-1)^2}{4X} + (\pi-a)\frac{(Y-1)^2}{4Y} = 0
        \label{d_lambda_n}
        \end{equation}
        on account of $(X, Y) \in \mathcal D_a$ and \eqref{eq: D_a definition}.
        Similarly, we can compute
        \begin{equation}
        \norm{\phi_\lambda}^2 = \norm{\psi_\lambda}^2 = \frac{|\lambda|^2}{2}\left(|\sinh(\lambda(\pi-a))|^2H_{\Re\lambda}(a) + |\sinh(\lambda a)|^2H_{\Re \lambda}(\pi-a)\right),
        \label{eq: Eigenfunction norm}
        \end{equation}
        where we define
        \begin{equation}
        H_\omega(t) \coloneqq \begin{cases}
            \frac{\sinh(2\omega t)}{\omega}, & \omega \neq 0, \\
            2t, & \omega = 0.
        \end{cases}
        \label{eq: Definition of H}
        \end{equation}
        Since $\Re \lambda_n^{(j)} \to r$, we acquire
        \[
        \lim_{n\to \infty}\frac{\norm{\psi_n^{(j)}}^2}{|\lambda_n^{(j)}|^2} = \frac{1}{8}\left(\frac{|Y-1|^2}{|Y|}H_s(a) + \frac{|X-1|^2}{|X|}H_s(\pi-a)\right) > 0.
        \]
        It follows that $\norm{P_n^{(j)}} \to +\infty$.

        \item It is left to show that with $\alpha \in \mathbb C \setminus (\mathcal E_a \cup \{\pm 2\})$, there exists a normalization of the eigenvectors $\psi_\lambda$ and $\phi_\lambda$ that gives both biorthogonality and the Bessel property. Note that we need not consider generalized eigenvectors thanks to Lemma \ref{L: Double eigenvalues are exceptional}. In fact, we only need to show that there is a normalization by constants $b_\lambda$, $c_\lambda$ satisfying the Bessel property such that the pairing constants $\kappa_\lambda \coloneqq \inner{b_\lambda\phi_\lambda}{c_\lambda\psi_\lambda}$ satisfy
        \[
        \inf_{\lambda\in\sigma(A)}|\kappa_\lambda| > 0.
        \]
        If that is the case, then
        \[
        \left\{\frac{b_\lambda}{\kappa_\lambda}\phi_\lambda\right\}_{\lambda \in \sigma(A)}, \quad \left\{c_\lambda\psi_\lambda\right\}_{\lambda\in\sigma(A)}
        \]
        are the desired biorthogonal Bessel and complete sequences, compliments by Lemma \ref{L: Total condition}. We claim that
        \[
        b_\lambda = c_\lambda = \frac{1}{\norm{\psi_\lambda}} = \frac{1}{\norm{\phi_\lambda}}
        \]
        is such normalization.

        Let us start by proving the boundedness of $\kappa_\lambda$. Assume now for a contradiction that there is a sequence $\{\lambda_n\}_{n=1}^\infty$ such that $\kappa_{\lambda_n} \to 0$. Notice that the transformed characteristic equation
        \[
        G(\lambda) = (2+\alpha)\mathrm{e}^{2\lambda\pi} - \alpha\mathrm{e}^{2\lambda a} - \alpha \mathrm{e}^{2\lambda(\pi-a)} + \alpha - 2 = 0
        \]
        forces
        \begin{equation}
        -C \leq \Re \lambda \leq C, \quad C \equiv C(a, \alpha) > 0, \quad \alpha \neq \pm 2.
        \label{eq: Real part restriction}
        \end{equation}
        Therefore, we can choose a subsequence $\{\lambda'_n\}_{n=1}^\infty$ such that both the real part and the phases converge:
        \[
        \lim_{n\to\infty}\mathrm{e}^{2a\lambda_n'} \eqqcolon X, \quad \lim_{n\to\infty}\mathrm{e}^{2(\pi-a)\lambda_n'} \eqqcolon Y.
        \]
        It clearly holds that 
        \begin{equation}
            |X|^{\pi-a} = |Y|^a.
            \label{eq: Modulus condition}
        \end{equation}
        Similarly to \eqref{eq: d_lambda} and \eqref{d_lambda_n}, we get
        \begin{equation}
        a\frac{(X-1)^2}{4X} + (\pi-a)\frac{(Y-1)^2}{4Y} = 0.
        \label{eq: Equation for X, Y}
        \end{equation}
        At the same time, the characteristic equation yields
        \begin{equation}
        \alpha = -2\frac{(XY-1)}{(X-1)(Y-1)}.
        \label{eq: Equation for X, Y, alpha}
        \end{equation}
        as long as we show that indeed $X, Y \neq 1$. By \eqref{eq: Equation for X, Y}, the only possible way for this to happen is $X = Y = 1$. We have already found in \eqref{eq: d_lambda} and \eqref{eq: Eigenfunction norm} that
        \[
        \kappa_\lambda = \frac{|\inner{\phi_\lambda}{\psi_\lambda}|}{\norm{\phi_\lambda}^2} = 2\frac{a\sinh^2(\lambda(\pi-a)) + (\pi-a)\sinh^2(\lambda a)}{|\sinh(\lambda(\pi-a))|^2H_{\Re\lambda}(a) + |\sinh(\lambda a)|^2H_{\Re \lambda}(\pi-a)},
        \]
        where $H_\omega(t)$ is defined by \eqref{eq: Definition of H}. Therefore, if we set
        \[
        \zeta_n \coloneqq - \frac{\sinh(\lambda_n'a)}{\sinh(\lambda_n'(\pi-a))},
        \]
        we can rewrite
        \begin{equation}
        |\kappa_{\lambda_n'}| = 2\frac{|a + (\pi-a)\zeta_n^2|}{H_{\Re \lambda_n'}(a) + |\zeta_n|^2H_{\Re \lambda_n'}(\pi-a)}.
        \label{eq: Formula for kappa}
        \end{equation}
        If $X = Y = 1$, we have
        \[
        a\lambda_n' = \mathrm{i}\pi m_n + z_n, \quad (\pi-a)\lambda_n' = i\pi k_n + w_n, \quad m_n, k_n \in \mathbb Z, \quad z_n, w_n \to 0.
        \]
        Due to the periodicity, the characteristic function can be written as
        \[
        \sinh(z_n + w_n) + \alpha \sinh z_n\sinh w_n = 0.
        \]
        It follows that $w_n = -z_n + \mathcal O(\max\{|z_n|^2, |w_n|^2\})$ as $n \to \infty$, and hence $\zeta_n^2 \to 1$. Since we also have $\Re \lambda_n' \to 0$, formula \eqref{eq: Formula for kappa} yields
        \[
        \kappa_{\lambda_n'} \to \frac{2|a + (\pi-a)|}{2a + 2(\pi-a)} = 1,
        \]
        which contradicts the assumption $\kappa_{\lambda_n'} \to 0$. Therefore, it must hold that $X, Y \neq 1$. Together with  \eqref{eq: Modulus condition}, \eqref{eq: Equation for X, Y} and \eqref{eq: Equation for X, Y, alpha}, this means that $(X, Y) \in \mathcal D_a$ and $\alpha \in \mathcal E_a$ -- a contradiction with the assumption.

        Finally, we turn our attention to the Bessel property. We first introduce the isometry
        \[
        U:\mathcal H \to L^2((0,\pi), \mathbb C^2) : \begin{pmatrix}
            u \\ v
        \end{pmatrix} \mapsto \frac{1}{\sqrt{2}}\begin{pmatrix}
            v + u' \\
            v - u'
        \end{pmatrix}.
        \]
        It is easily computed that under this transformation, the eigenvectors become
        \[
        U\psi_\lambda \eqqcolon \lambda \begin{pmatrix} p_\lambda \\ q_\lambda\end{pmatrix},
        \]
        where
        \begin{equation}
        (p_\lambda(x), q_\lambda(x)) = \frac{1}{\sqrt{2}}\begin{cases}
            \sinh(\lambda(\pi-a))(\mathrm{e}^{\lambda x}, -\mathrm{e}^{-\lambda x}), & 0 \leq x \leq a,\\
            -\sinh(\lambda a)(\mathrm{e}^{\lambda(x-\pi)}, - \mathrm{e}^{-\lambda(x-\pi)}), & a \leq x \leq \pi.
        \end{cases}
        \label{eq: Transformed Eigenvector}
        \end{equation}
        By \eqref{eq: Eigenfunction norm}, we have $\norm{\psi_\lambda} = |A_\lambda|^2H_{\Re \lambda}(a) + |B_\lambda|^2H_{\Re\lambda}(\pi-a)$, where we define
        \[
        A_\lambda \coloneqq \frac{\lambda}{\sqrt{2}}\sinh(\lambda (\pi-a)), \quad B_\lambda \coloneqq \frac{\lambda}{\sqrt{2}}\sinh(\lambda a),
        \]
        and the functions $\lambda \mapsto H_{\Re \lambda (a)}$ and $\lambda \mapsto H_{\Re \lambda}(\pi -a)$ are bounded uniformly from below due to the restriction \eqref{eq: Real part restriction}. Therefore, there is a constant $K$ such that
        \[
        \frac{| A_\lambda|}{\norm{\psi_\lambda}} \leq K, \quad \frac{|B_\lambda|}{\norm{\psi_\lambda}} \leq K.
        \]
        In other words, every transformed normalized eigenvector consists of a combination of the four exponentials \eqref{eq: Transformed Eigenvector} with uniformly bounded coefficients. It is thus sufficient to prove that $\{\mathrm{e}^{\pm\lambda x}\}_{\lambda \in \sigma(A)}$ are Bessel in both $L^2(0, a)$ and $L^2(a, \pi)$. That is essentially the statement of the following auxiliary lemma:

        \begin{lemma}
            Let $a, b \in \mathbb R$, $a<b$, and let $\seq{\lambda} \subset \mathbb C$ satisfy
            \[
            \sup_n|\Re \lambda_n| < \infty \quad \text{and} \quad \sup_{s \in \mathbb R} \left|\{n \mid \Im \lambda_n \in [s, s+1] \}\right| < +\infty.
            \]
            Then $\{\mathrm{e}^{\lambda_n x}\}_{n=1}^\infty$ is a Bessel sequence in $L^2(a, b)$.
        \end{lemma}
        \begin{proof}
            Let us denote $I \coloneqq (-L, L)$, where $L > 0$ is large enough to ensure $(a, b) \subset I$. For $f \in L^2(-L, L)$, define
            \[
            F(k) \coloneqq \int_{-L}^L f(x) \mathrm{e}^{-\mathrm{i}kx}\dd x.
            \]
            Then $F$ is of exponential type at most $L$ and
            \[ 
            \norm{F}_{L^2(\mathbb R)} = 2\pi \norm{f}_{L^2(-L,L)}
            \]
            by Plancherel's theorem.
            Let us put $\zeta_n \coloneqq \mathrm{i}\lambda_n$. This ensures that
            \[
            \Re \zeta_n = -\Im\lambda_n, \quad \Im \zeta_n = \Re \lambda_n.
            \]
            By the remark succeeding \cite[Theorem 2.3.17]{Young}, it follows that
            \[
            \sum_{n=1}^\infty |F(\zeta_n)|^2 \leq c\norm{F}_{L^2(\mathbb R)}.
            \]
            Then we can readily calculate
            \begin{align*}
                \sum_{n=1}^\infty\left|\inner{\mathrm{e}^{\lambda_n\cdot}}{f}\right|^2 &= \sum_{n=1}^\infty \left|\int_{-L}^Lf(x)\mathrm{e}^{\lambda_n x}\dd x\right|^2 = \sum_{n=1}^\infty\left|F(\zeta_n)\right|^2 \leq c \norm{F}_{L^2(\mathbb R)} \\&= 2\pi c\norm{f}_{L^2(-L,L)}.
            \end{align*}
            This precisely means that the sequence of exponentials is Bessel in $L^2(-L, L)$ and thus also in $L^2(a, b)$ by standard arguments.
        \end{proof}
        The last ingredient missing is to prove that $\sigma(A)$ satisfies the condition 
        \[
        \sup_{s\in\mathbb R}\left|\{\lambda \mid \Im \lambda \in [s, s+1] \}\right| < +\infty.
        \]
        This, however is a direct consequence of Theorem \ref{BoundednessOfEigenvalues}.

        The fact that also $\phi_\lambda$ are Bessel after normalization follows from the same arguments applied to the transformation
        \[
        U\phi_\lambda = \overline{\lambda}\begin{pmatrix}
            -\overline{q_\lambda} \\ -\overline{p_\lambda}
        \end{pmatrix}.
        \]
        This concludes the proof. \qedhere
 \end{enumerate}
\end{proof}

The proof of Theorem \ref{Criterion} is hence complete. Let us also highlight the following result from the proof of Theorem \ref{T: Irrational Riesz criterion}:
\begin{corollary}
    Let $a \in (0, \pi)$ with $a \notin \pi\mathbb Q$ and $\alpha \in \mathbb C$. Then $A(a, \alpha)$ has at most two algebraically double eigenvalues.
\end{corollary}

\section{Compact star graph}\label{CompactStarGraph}
\subsection{Wave equation on a compact star graph}
Following the footsteps and the notation of Krejčiřík and Royer in \cite{KrRoy}, we consider a metric graph $\Gamma$ consisting of $n$ copies of the compact interval $[0, \pi]$ connected at the central vertex. We set the lengths equal to each other to avoid introducing too many new constants to the model. For $n = 2$ this will be, up to scaling, equivalent to the case $a= \pi/2$ of the problem considered previously. 


To properly construct the corresponding Hilbert space setting, we first set
\[
L^2(\Gamma) \coloneqq \left(L^2(0, \pi)\right)^n
\]
with the standard inner product extension to a Cartesian product of Hilbert spaces. Next, let
\[
H^k(\Gamma^\ast) \coloneqq \left(H^k(0, \pi)\right)^n, \quad \dot H_0^k(\Gamma^\ast) \coloneqq \left\{u \in H^k(\Gamma^\ast) \mid (\forall j \in \{1, \dots, n\})(u_j(\pi) = 0)\right\}.
\]
The equality $u_j(\pi) = 0$ is to be understood in the sense of the absolutely continuous representative of the equivalence class $u_j \in H^k(0, \pi)$.

For $u \in H^1(\Gamma^\ast)$, we denote $u_j$ its components with $j \in \{1, \dots, n\}$. $u$ is said to be continuous at $0$ if there exist continuous representatives $u_j$ of the equivalence classes of each component such that $\left(\forall j, k \in \{1, \dots, n\}\right)\left(u_j(0) = u_k(0)\right)$. The common value is denoted as $u(0)$.

To obtain the continuity condition at the vertex in accordance with the previous model, we set
\[
\dot H_0^1(\Gamma) \coloneqq \left\{u \in \dot H_0^1(\Gamma^\ast) \mid u \text{ is continuous at } 0\right\}.
\]
Finally, we will work with the Hilbert space
\[
\mathcal H \coloneqq \dot H_0^1(\Gamma) \times L^2(\Gamma).
\]
The wave operator with Dirac damping at the central vertex of $\Gamma$ is then defined as follows:
\begin{gather*}
\Dom A_n(\alpha) \coloneqq \left\{\psi = \begin{pmatrix}
    u \\ v
\end{pmatrix} \in \left(\dot H_0^1(\Gamma) \cap H^2(\Gamma^\ast)\right) \times \dot H_0^1(\Gamma) \mid \sum_{j=1}^n u'_j(0) = \alpha v(0) \right\}, \\
A_n(\alpha)\begin{pmatrix}
    u \\ v
\end{pmatrix} \coloneqq \begin{pmatrix}
    v \\ u''
\end{pmatrix},
\end{gather*}
where the derivatives are understood by components.

\subsection{The characteristic function}
Computing the resolvent, one can, analogously to Section \ref{SecResolvent}, derive the characteristic function. Solving $u'' - \lambda(\lambda + \alpha\delta)u = 0$ with $u_j(\pi) = 0$, $u_j'(\pi) = -1$ for all $j \in \{1, \dots, n-1\}$, we find
\[
u_n(x) = \frac{1}{\lambda}\left(\alpha\sinh(\lambda\pi) + (n-1) \cosh(\lambda\pi)\right)\sinh(\lambda x) + \frac{1}{\lambda}\sinh(\lambda\pi)\cosh(\lambda x).
\]

At $x = \pi$, we have
\begin{equation}
S_n(\lambda;\alpha) \coloneqq u_n(\pi)
= \frac{\sinh(\lambda\pi)}{\lambda}\left(n\cosh(\lambda\pi) + \alpha \sinh(\lambda\pi)\right).    \label{GraphCharacteristicFunction}
\end{equation}
It can be easily seen that for $n=2$ and after scaling, this is consistent with \eqref{DefinitionOfS} for the central damping $a = \pi/2$.

Since $A_n^{-1}$ is again compact, $A_n$ has purely discrete spectrum. We invoke Proposition \ref{IndexPole} to state the following analogy of Theorem \ref{CharacteristicFunction}. 

\begin{proposition}\label{GraphSProposition}
    $\lambda \in \mathbb C$ is an eigenvalue of $A_n(\alpha)$  if and only if it is a root of the entire function \eqref{GraphCharacteristicFunction}. Furthermore, the index of $\lambda \in \sigma_p(A_n(\alpha))$ is precisely its root multiplicity.
\end{proposition}

The situation here is significantly easier since we abandoned the parameter $a$. We simply have
\begin{align}
    \lambda S_n(\lambda, \alpha) = \frac{\mathrm{e}^{2\lambda\pi}}{4}\left(n(1-e^{-4\lambda\pi}) + \alpha(1-2\mathrm{e}^{-2\lambda\pi} + \mathrm{e}^{-4\lambda\pi})\right) = -\frac{\mathrm{e}^{2 \lambda\pi}}{4}P_{n,\alpha}\left(\mathrm{e}^{-2\lambda\pi}\right),
\end{align}
where we define the polynomial $P_{n,\alpha}$ as
\begin{equation} \label{GraphPolynomial}
    P_{n,\alpha}(z) = (n-\alpha)z^2 + 2\alpha z -(n + \alpha).
\end{equation}
Here we can observe for the first time why $\alpha = \pm n$ are somewhat special values.

Supposing $\alpha \neq n$, the roots are
\[
\zeta_1 = 1,\quad \zeta_2 = \frac{\alpha + n}{\alpha - n},
\] with the corresponding eigenvalues
\begin{align}
    \lambda_{1,k} &= \mathrm{i}k, \quad k \in \mathbb Z \setminus\{0\}, \nonumber \\
    \lambda_{2, k} &= -\frac{1}{2\pi}\left(\ln\left|\frac{\alpha +n}{\alpha-n}\right| + \mathrm{i}(\theta + 2\pi k)\right), \quad k \in \mathbb Z,
    \label{GraphEigenvalues}
\end{align}
now supposing also $\alpha \neq -n$ and denoting $\theta$ the argument of $\zeta_2$. With the help of Proposition \ref{GraphSProposition}, it is easy to verify that $\lambda_{1,k}$ are of algebraic (and geometric) multiplicity $n-1$ while $\lambda_{2,k}$ are algebraically (and geometrically) simple.

\subsection{Eigenvectors}
By Proposition \ref{GraphSProposition} it is clear that no eigenvalue has index greater than $1$. Therefore, all root vectors are in fact eigenvectors. We will once again use Theorem \ref{RieszCriterion} to show in which cases the eigenvectors of $A_n(\alpha)$ form a Riesz basis in $\mathcal H$. 

As one can show analogously to Proposition \ref{AdjointProposition} that $A_n^\ast(\alpha) = -A_n\left(-\overline\alpha\right)$, the biorthogonal sequence is provided by the eigenvectors of $A_n^\ast(\alpha)$ similarly as in the prequel.

Solving the eigenvalue problem $A_n(\alpha) \psi = \lambda \psi$, we obtain eigenfunctions of the form
\[
\psi = \begin{pmatrix}
    u \\ \lambda u
\end{pmatrix}, \quad u_j(x) \coloneqq \sinh(\lambda(\pi-x)), \,\, j \in \{1, \dots, n\}. 
\]

To show the Bessel property, one can again proceed similarly to Cox and Henrot in \cite{CH} and use the explicit form of the eigenvalues \eqref{GraphEigenvalues}. Details are omitted. 

Perhaps more interesting is the Livšic criterion, which we apply to $A_n^{-1}$. First, we will find the trace of the real part.

\begin{proposition}\label{GraphTrace}
Let $\alpha \in \mathbb C$, $n \in \mathbb N$. Then $\Tr\Re A_n^{-1}(\alpha) = -\pi\frac{\Re\alpha}{n}$.
\end{proposition}

\begin{proof}
    Consider the equation
\[
A_n(\alpha)\begin{pmatrix}
    u \\ v
\end{pmatrix} = \begin{pmatrix}
    f \\ g
\end{pmatrix},
\]
yielding $v = f$, $u'' = g$, and $\sum_{j=1}^nu_j'(0) = \alpha v(0)$. The solution is of the form
\[
u_j(x) = \int_0^x(x-t)g_j(t)\dd t - \frac{x}{\pi}\int_0^\pi(\pi-t)g_j(t) \dd t + \left(\alpha f(0) + \frac{1}{\pi}\int_\Gamma (\pi-t)g(t) \dd t\right)\frac{x-\pi}{n},
\]
where $\int_\Gamma (\pi-t)g(t)\dd t \coloneqq \sum_{k=1}^n \int_{0}^\pi(\pi-t)g_k(t) \dd t$.
Similarly, for the equation
\[
A^\ast_n(\alpha)\begin{pmatrix}
    u \\ v
\end{pmatrix} = \begin{pmatrix}
    f \\ g
\end{pmatrix},
\]
we obtain $v = -f$, $u'' = -g$, and $\sum_{j=1}^nu_j'(0) = \overline{\alpha}f(0)$.
This leads to
\[
u_j(x) = -\int_0^x(x-t)g_j(t)\dd t + \frac{x}{\pi}\int_0^\pi(\pi-t)g_j(t) \dd t + \left(\overline\alpha f(0) - \frac{1}{\pi}\int_\Gamma (\pi-t)g(t) \dd t\right)\frac{x-\pi}{n}.
\]
Overall, for any $\begin{pmatrix}
    f \\ g
\end{pmatrix} \in \mathcal H$, we have found
\[
\Re A_n^{-1}\begin{pmatrix}
    f \\ g
\end{pmatrix} = \frac{1}{2}\left(A_n^{-1} +(A_n^{-1})^\ast\right)\begin{pmatrix}
    f \\ g
\end{pmatrix} = \frac{\Re \alpha}{n}f(0)\begin{pmatrix}
    \rho \\ 0
\end{pmatrix}, \quad \rho_j(x) \coloneqq x- \pi.
\]

To find the trace, we choose an arbitrary orthonormal basis of $\mathcal H$ that contains the normalized spanning vector of $\Ran \Re A_n^{-1}$ of the form $\psi_0 \coloneqq 1/\sqrt{n\pi}(\rho, 0)$.
Then we compute as follows
\[
    \Tr\Re A_n^{-1}(\alpha) = \frac{1}{n\pi}\inner{\begin{pmatrix}
        \rho \\ 0
    \end{pmatrix}}{\Re A_n^{-1}(\alpha) \begin{pmatrix}
        \rho \\ 0
    \end{pmatrix}} = -\frac{\Re \alpha}{n^2}\norm{\begin{pmatrix}
        \rho \\ 0
    \end{pmatrix}}^2_\mathcal H = -\frac{\Re \alpha}{n}\pi.
    \qedhere
\]
\end{proof}

Let us now calculate the series
\[
\sum_{\lambda\in\sigma(A_n)}\Re \frac{1}{\lambda} = \sum_{k\in\mathbb Z}\Re\frac{1}{\lambda_{2,k}}
\]
using \eqref{GraphEigenvalues} for $\alpha \neq \pm n$. Computing the real part, we have
\[
\frac{1}{\lambda_{2,k}} = - 2\pi \frac{\ln\left|\frac{\alpha + n}{\alpha - n}\right| - \mathrm{i}(\theta + 2\pi k)}{\ln^2\left|\frac{\alpha + n}{\alpha - n}\right| + (\theta + 2\pi k)^2} \implies \Re \frac{1}{\lambda_{2,k}} = -\frac{\ln\left|\zeta_2\right|}{2\pi}\frac{1}{\frac{\ln^2\left|\zeta_2\right|}{4\pi^2} + \left(k + \frac{\theta}{2\pi}\right)^2}.
\]
Invoking \eqref{TheSeries} with $\gamma \coloneqq \theta/2\pi$ and $\beta \coloneqq \ln\left|\zeta_2\right|/2\pi$, we obtain
\begin{align}
\sum_{k\in \mathbb Z}\Re \frac{1}{\lambda_{2, k}} &= -\frac{\pi}{2} \frac{\sinh(\ln|\zeta_2|)}{\cosh^2\left(\frac{\ln|\zeta_2|}{2}\right) - \cos^2(\theta/2)} = \pi \frac{1-|\zeta_2|^2}{|\zeta_2|^2 - 2 \Re \zeta_2 + 1} = \pi \Re \frac{\zeta_2 + 1}{1-\zeta_2} \nonumber\\
&= -\pi\frac{\Re\alpha}{n}.
\label{GraphSpectralTrace}
\end{align}


On the other hand, for $\alpha = \pm n$, it follows directly from \eqref{GraphPolynomial} and \eqref{GraphEigenvalues} that $A_n(\alpha)$ has purely imaginary spectrum. Therefore,
\begin{equation}
\sum_{\lambda \in \sigma(A_n(\pm n))} \Re \frac{1}{\lambda} = 0.
\label{GraphCriticalDamping}
\end{equation}

Finally, comparing results \eqref{GraphSpectralTrace} and \eqref{GraphCriticalDamping} with Proposition \ref{GraphTrace}, we conclude with the following generalization of Theorem \ref{RieszBasisCriterion} that clarifies the appearance of the peculiar constant $\alpha = \pm 2$ in the model of Bamberger, Rauch, and Taylor.

\begin{theorem}\label{T: GraphRieszCriterion}
    Let $\alpha \in \mathbb C$, $n \in \mathbb N$. Then the eigenvectors of $A_n(\alpha)$ form a Riesz basis in $\mathcal H$ if and only if $\alpha \neq \pm n$.
\end{theorem}

\section*{Declaration}
The author was supported by the grant no.~26-21940S of the Czech Science Foundation.


%
\providecommand{\bysame}{\leavevmode\hbox to3em{\hrulefill}\thinspace}
\providecommand{\MR}{\relax\ifhmode\unskip\space\fi MR }
\providecommand{\MRhref}[2]{%
  \href{http://www.ams.org/mathscinet-getitem?mr=#1}{#2}
}
\providecommand{\href}[2]{#2}

\end{document}